\documentclass[12pt]{article}
\usepackage{amssymb,amsfonts,amsmath,amsthm,amsrefs,hyperref}
\usepackage{geometry}
\newcommand{\ii}{\mathfrak{i}}
\newcommand{\Z}{\mathbb{Z}}

\newcommand{\R}{\mathbb{R}}
\newcommand{\C}{\mathbb{C}}
\newcommand{\Cp}{\C^{\times}}
\newcommand{\Rp}{\R^{\times}}
\newcommand{\N}{\mathbb{N}}

\newcommand{\T}{\mathbb{T}}

\newcommand{\8}{\infty}

\newcommand{\spa}{\mathrm{span}}

\newcommand{\re}{\mathrm{Re~}}

\newcommand{\diag}{\mathrm{diag}}

\newcommand{\Int}{\mathrm{int}}

\newcounter{dummy} \numberwithin{dummy}{section}
\newtheorem{theorem}[dummy]{Theorem}
\newtheorem{lemma}[dummy]{Lemma}

\newtheorem{proposition}[dummy]{Proposition}
\newtheorem{corollary}[dummy]{Corollary}
\newtheorem{question}[dummy]{Question}
\theoremstyle{remark}
\newtheorem{remark}[dummy]{Remark}
\newtheorem{example}[dummy]{Example}

\begin{document}

\title{Principal Minor Assignment, Isometries of Hilbert Spaces, Volumes of Parallelepipeds and Rescaling of Sesqui-holomorphic Functions}
\author{Eugene Bilokopytov\footnote{Email address bilokopi@myumanitoba.ca, erz888@gmail.com.}}
\maketitle

\begin{abstract}In this article we consider the following equivalence relation on the class of all functions of two variables on a set $X$: we will say that $L,M: X\times X\to \C$ are rescalings if there are non-vanishing functions $f,g$ on $X$ such that $M\left(x,y\right)=f\left(x\right)g\left(y\right) L\left(x,y\right)$, for any $x,y\in X$. We give criteria for being rescalings when $X$ is a topological space, and $L$ and $M$ are separately continuous, or when $X$ is a domain in $\C^{n}$ and $L$ and $M$ are sesqui-holomorphic.

A special case of interest is when $L$ and $M$ are symmetric, and $f=g$ only has values $\pm 1$. This relation between $M$ and $L$ in the case when $X$ is finite (and so $L$ and $M$ are square matrices) is known to be characterized by the equality of the principal minors of these matrices. We extend this result to the case when $X$ is infinite. As an application we characterize restrictions of isometries of Hilbert spaces on weakly connected sets as the maps that preserve the volumes of parallelepipeds spanned by finite collections of vectors.\medskip


MSC2010 15A15, 46E22, 51M25.
\end{abstract}

\textbf{This is an extended version of the article by the same title \href{https://authors.elsevier.com/a/1ZIKK5YnCbdlb}{published} in Linear Algebra and its Applications.}

\section{Introduction}

The following well-known open problem in algebra is called Principal Minor Assignment Problem.

\begin{question}
What are the necessary and sufficient conditions for a collection of $2^{n}$ complex numbers to be the set of the principal minors of a $n\times n$ matrix?
\end{question}

For some recent developments in relation to this question see e.g. \cite{oeding} and \cite{rkt}. An adjacent to the Principal Minor Assignment Problem is the following easier question.

\begin{question}\label{q}
What is the relationship between two $n\times n$ matrices having equal corresponding principal minors of all orders?
\end{question}

It turns out that the answer very much depends on certain additional assumptions about the matrices (see e.g. \cite{loewy}). In particular, in the class of complex symmetric matrices the following characterization holds.

\begin{theorem}[\cite{es}]\label{simila}
If two complex symmetric $n\times n$ matrices $L$ and $M$ have equal corresponding principal minors of all orders, then there is a diagonal matrix $D$ of order $n$ with diagonal entries in $\left\{-1,1\right\}$ such that $L=DMD^{-1}$.
\end{theorem}

The proof of this fact given in \cite{es} and \cite{rkt} is mostly combinatorial and heavily exploits an auxiliary object -- the \emph{graph of a matrix}. Namely, the graph of a symmetric $n\times n$ matrix $L=\left[l_{ij}\right]_{i,j=1}^{n}$ is a graph with vertices in $\left\{1,...,n\right\}$ and $\left(i,j\right)$ being an edge if $l_{ij}\ne 0$. Hence, the indices $i$ and $j$ are treated as points of a certain space, rather than numbers. Also, the proof does not rely much on finiteness of $\left\{1,...,n\right\}$, which motivates us to consider the following type of an object.

Let $X$ be a set. We will call complex-valued functions defined on $X\times X$ \emph{bi-functions}. Note that a bi-function on a finite set $X=\left\{1,...,n\right\}$ is a square $n\times n$ matrix. One can introduce the notion of the graph of a bi-function in an analogous way to the graph of a matrix. In order to state an analogue of Theorem \ref{simila} we need to find the concepts corresponding to the principal minors and the diagonal similarity. Namely, for a bi-function $L$ on $X$ and $x_{1},...,x_{n}\in X$ denote $\det_{L}\left(x_{1},...,x_{n}\right)=\det\left[L\left(x_{i},x_{j}\right)\right]_{i,j=1}^{n}$. A notion analogous to the diagonal similarity is the following. We will say that a non-vanishing function $f$ on $X$ \emph{reciprocally rescales} a bi-function $L$ on $X$ to a bi-function $M$, if $M\left(x,y\right)=f\left(x\right)\frac{1}{f\left(y\right)} L\left(x,y\right)$, for all $x,y\in X$. We then say that $L$ and $M$ are \emph{reciprocal rescalings}. It is easy to see that this definition generalizes the condition that appears in Theorem \ref{simila}. In fact, the theorem holds in this infinite context verbatim.

\begin{theorem}\label{simil}
Symmetric bi-functions $L$ and $M$ on a set $X$ are reciprocal rescalings of each other if and only if $\det_{L}\left(x_{1},...,x_{n}\right)=\det_{M}\left(x_{1},...,x_{n}\right)$, for any $x_{1},...,x_{n}\in X$.
\end{theorem}

More generally, for non-vanishing functions $f,g:X\to\C$ we will say that the pair $\left(f,g\right)$ \emph{rescales} $L$ to $M$ if $M\left(x,y\right)=f\left(x\right)g\left(y\right) L\left(x,y\right)$, for any $x,y\in X$. Reciprocal rescaling corresponds to the case when $g=\frac{1}{f}$, but we can choose other relations between $f$ and $g$. In particular, if $g=f$ (or $g=\overline{f}$) we will say that $f$ \emph{symmetrically (or Hermiteanly)} rescales $L$ to $M$. Hermitean rescalings of \emph{positive semi-definite kernels} play a role in studying Multiplication Operators on Reproducing Kernel Hilbert Spaces (see e.g. \cite{arsw}).

One can ask if there is a way to ascertain if two bi-functions $L$ and $M$ are rescalings without referring to $f$ and $g$. It is easy to see that if $L$ and $M$ are rescalings, then $$M\left(x,y\right)M\left(y,z\right)L\left(x,z\right)L\left(y,y\right)=L\left(x,y\right)L\left(y,z\right)M\left(x,z\right)M\left(y,y\right),$$ for every $x,y,z$ in $X$. It is also easy to see that the converse holds if there is $y\in X$ such that $L\left(\cdot,y\right)$ and $L\left(y,\cdot\right)$ do not vanish. However, in general the converse does not hold.

Various questions related to rescaling of bi-functions, their graphs and minors (including the proof of Theorem \ref{simil} and its variants) are considered in sections \ref{bf} and \ref{minors}.

Since the underlying set $X$ is no longer confined to the finite world, one can add structure on $X$ into consideration. In this article we deal with two examples of such structures. In Section \ref{bf} we assume that $X$ is a topological space and consider bi-functions that satisfy some continuity conditions. We study how the topological properties of $X$ impact the graph theoretical properties of the graphs of such bi-functions, and under some additional restrictions we show that the equality above is a sufficient condition for $L$ and $M$ to be rescalings (see Theorem \ref{qr}, and also Example \ref{exa} that demonstrates that the conditions of this theorem are essential).

In Section \ref{resq} we consider the case when $X$ is a domain in $\C^{n}$ and study holomorphic and sesqui-holomorphic bi-functions on $X$. Since such bi-functions are determined by their values on relatively small sets, the corresponding rescaling criteria are much more succinct (see Proposition \ref{rigs} and Theorem \ref{rrig}).\medskip

In Section \ref{geoin} we focus on the geometric interpretation of Theorem \ref{simila} in the case when we restrict to the real positive definite matrices. These matrices are the Gram matrices of certain collections of vectors in $\R^{n}$ and the determinant of the Gram matrix is the square of the volume of the parallelepiped spanned by that collection. Hence, a direct consequence of Theorem \ref{simila} is the following fact.

\begin{corollary}\label{para}
Two parallelepipeds with equal volumes of the corresponding faces are isometric.
\end{corollary}

We give a mostly geometric proof of this result independent of Theorem \ref{simila}. Joining the geometrical and the topological approaches allows us to characterize restrictions of isometries of Hilbert spaces on weakly connected sets as the maps that preserve the volumes of parallelepipeds spanned by finite collections of vectors (see Theorem \ref{iso}).\medskip

\textbf{Some notations and conventions. }Let $\Cp=\C\backslash\left\{0\right\}$, $\Rp=\R\backslash\left\{0\right\}$ and let $\T=\left\{\lambda\in\C,\left|\lambda\right|=1\right\}$. We will also denote the imaginary unit number by $\ii$ in order to reserve the letter $i$ for indexation (typically by integers). If $X$ is a set, we will view the Kronecker's $\delta$ as a function defined on $X\times X$, and for a fixed $x\in X$ $\delta_{x}$ will be viewed as a function on $X$. Finally, $Fin\left(X\right)$ will be the set of all finite subsets of $X$. If $\mu:Fin\left(X\right)\to \C$, we will adopt the following abuse of notations: for $x_{1},...,x_{n}\in X$, we will denote $\mu\left(x_{1},...,x_{n}\right)=\mu\left(\left\{x_{1},...,x_{n}\right\}\right)$, if $x_{1},...,x_{n}$ are all distinct, and $\mu\left(x_{1},...,x_{n}\right)=0$ otherwise.

\section{Bi-functions}\label{bf}

The main focus of this article is on the (abstract) functions of two variables. Let $X$ be a set. We will call functions defined on $X\times X$ \emph{bi-functions} and usually use capital letters to denote them, in order to distinguish from the ``usual'' function, i.e. scalar-valued functions defined on $X$. Note that a bi-function on a finite set $X=\left\{1,...,n\right\}$ is a $n\times n$ matrix. Consider the simplest examples of bi-functions on a general set: if $f,g:X\to\C$, define $f\otimes g:X\times X\to \C$ by $\left[f\otimes g\right]\left(x,y\right)=f\left(x\right)g\left(y\right)$ and $\diag f:X\times X\to \C$ by $\diag f\left(x,y\right)=\delta_{x,y}f\left(x\right)$. Note that $f_{1}\otimes g_{1}\cdot f_{2}\otimes g_{2}=f_{1}f_{2}\otimes g_{1}g_{2}$ and $\diag f_{1}\diag f_{2}=\diag \left(f_{1}f_{2}\right)$. Consider also a transition from a bi-function to a function: if $L:X\times X\to \C$, define the \emph{diagonal function} $\widehat{L}:X\to \C$ by $\widehat{L}\left(x\right)=L\left(x,x\right)$, for $x\in X$. We will say that $L$ is \emph{non-degenerate} if $\widehat{L}$ does not vanish. Although purely technical, this condition is very useful in our considerations. See Remark \ref{ma} for a justification of the term ``non-degenerate''.

For a bi-function $L$ on $X$ define bi-functions $L'$ and $L^{*}$ by $L'\left(x,y\right)=L\left(y,x\right)$, for $x,y\in X$, and $L^{*}=\overline{L'}$. We will say that $L$ is \emph{symmetric (Hermitean)} if $L=L'$ ($L=L^{*}$). Clearly, $f\otimes f$ and $\diag f$ are symmetric and $f\otimes \overline{f}$ is Hermitean, for any $f:X\to\C$. Note that if $L$ is Hermitean, then $\widehat{L}$ is real-valued.\medskip

\textbf{Rescalings. } We will say that a pair $\left(f,g\right)$ of non-vanishing functions on $X$ \emph{rescales} a bi-function $L$ on $X$ to a bi-function $M$, if $M=f\otimes g L$. Note that in this case $L\left(x,y\right)=0\Leftrightarrow M\left(x,y\right)=0$, for $x,y\in X$, and $fg=\frac{\widehat{M}}{\widehat{L}}$, outside of $\left\{x\in X\left|\widehat{L}\left(x\right)=0\right.\right\}$. We will say that bi-functions $L$ and $M$ on $X$ are \emph{rescalings}, if there are functions $f,g:X\to\Cp$ such that that $\left(f,g\right)$ rescales $L$ to $M$.

It is easy to see that if $\left(f,g\right)$ rescales $L$ to $M$, then $\left(g,f\right)$ rescales $L'$ to $M'$, $\left(\overline{f},\overline{g}\right)$ rescales $\overline{L}$ to $\overline{M}$, $\left(\overline{g},\overline{f}\right)$ rescales $L^{*}$ to $M^{*}$, and $\left(\frac{1}{f},\frac{1}{g}\right)$ rescales $M$ to $L$. Also, $\left(\alpha f,\beta g\right)$ rescales $L$ to $\alpha\beta M$, for any $\alpha,\beta \in \Cp$, and in particular $\left(\lambda f,\lambda^{-1}g\right)$ rescales $L$ to $M$ for any $\lambda\in\Cp$. Hence, $f$ and $g$ are not uniquely determined by $L$ and $M$.

If $\left(f_{1},g_{1}\right)$ rescales $K$ to $L$ and $\left(f_{2},g_{2}\right)$ rescales $L$ to $M$, then $\left(f_{1}f_{2},g_{1}g_{2}\right)$ rescales $K$ to $M$. Hence, the relation of being rescalings is an equivalence relation. If $\left(f_{i},g_{i}\right)$ rescales $L_{i}$ to $M_{i}$, then $\left(\prod \limits_{i}f_{i},\prod \limits_{i}g_{i}\right)$ rescales $\prod \limits_{i}L_{i}$ to $\prod \limits_{i}M_{i}$.

Let $Y,Z$ be subsets of $X$ such that $Y\subset Z\subset X$ and let $f,g:Z\to\Cp$. We will say that $\left(f,g\right)$ rescales $L$ to $M$ on $Y$ if $f$ and $g$ rescale $\left. L\right|_{Y\times Y}$ to $\left. M\right|_{Y\times Y}$; in this case we will say that $L$ and $M$ are rescalings on $Y$. Being rescalings on $Y$ is an equivalence relation, which is weaker than being rescalings; if $L$ and $M$ are rescalings on $Z\supset Y$ then they are rescalings on $Y$. Let us now introduce equivalence relations between bi-functions, which are finer that being rescalings.\medskip

We will say that $f$ \emph{symmetrically} rescales $L$ to $M$ if $\left(f,f\right)$ rescales $L$ to $M$. In this case we will say that $L$ and $M$ are \emph{symmetric} rescalings. We will say that $f$ \emph{Hermiteanly} rescales $L$ to $M$ if $\left(f,\overline{f}\right)$ rescales $L$ to $M$. In this case we will say that $L$ and $M$ are \emph{Hermitean rescalings}. If $f$ symmetrically (or Hermiteanly) rescales $L$ to $M$ and $L$ is non-degenerate, then $f^{2}=\frac{\widehat{M}}{\widehat{L}}\ge0$ (or $\left|f\right|^{2}=\frac{\widehat{M}}{\widehat{L}}>0$), outside of $\left\{x\in X\left|\widehat{L}\left(x\right)=0\right.\right\}$. We will say that $f$ \emph{reciprocally} rescales $L$ to $M$ if $\left(f,\frac{1}{f}\right)$ rescales $L$ to $M$. In this case we will say that $L$ and $M$ are \emph{reciprocal rescalings}. Note that this concept is the analogue of the diagonal similarity of matrices. If $L$ and $M$ are reciprocal rescalings, then $\widehat{L}=\widehat{M}$; conversely if $\left(f,g\right)$ rescales $L$ to $M$, $L$ is non-degenerate and $\widehat{L}=\widehat{M}$, then $g=\frac{1}{f}$. For any $h:X\to\C$ any reciprocal rescaling of $\diag h$ is $\diag h$ itself. Hence, if $f$ reciprocally rescales $L$ to $M$, then $f$ also reciprocally rescales $L+\diag h$ to $M+\diag h$.

For a subgroup $\Gamma$ of $\Cp$ we will say that $L$ and $M$ are $\Gamma$-rescalings if there are $f,g:X\to \Gamma$ such that $\left(f,g\right)$ rescales $L$ to $M$. If in this case $f=g$, we will say that $L$ and $M$ are symmetric $\Gamma$-rescalings. In particular, if $\Gamma=\left\{-1,1\right\}$, we will call $L$ and $M$ (symmetric) $\pm 1$-rescalings. It is easy to see that if $L$ and $M$ are (symmetric) $\pm 1$-rescalings, then $L^{2}=M^{2}$ (and $\widehat{L}=\widehat{M}$). Conversely, if $L$ and $M$ are non-degenerate symmetric rescalings with $\widehat{L}=\widehat{M}$, then they are symmetric $\pm 1$-rescalings. Also, if $L$ and $M$ are non-degenerate symmetric rescalings with $\frac{\widehat{M}}{\widehat{L}}>0$, then they are symmetric $\Rp$-rescalings. In particular, this happens, when $L$ and $M$ are Hermitean and symmetric rescalings simultaneously. For more similar properties see Proposition \ref{cresc}.\medskip

\textbf{The graph of a bi-function. }The main idea that we borrow from \cite{es} and \cite{rkt} is the introduction of the following object. For a bi-function $L$ on $X$ let $X_{L}$ be the (undirected) graph with the set of vertices equal to $X$ and $\left(x,y\right)\in X\times X$ being an edge if either $L\left(x,y\right)\ne 0$ or $L\left(y,x\right)\ne 0$; for $Y\subset X$ define $Y_{L}=Y_{\left.L\right|_{Y\times Y}}$. The (graph-theoretical) components of $X_{L}$ are the equivalence classes of the minimal equivalence relation on $X$ that contains pairs $\left(x,y\right)\in X\times X$ such that $L\left(x,y\right)\ne 0$. Also, the components of $X_{L}$ can be viewed as ``degrees of freedom'' of rescaling of $L$, as the following proposition shows.

\begin{proposition}\label{components}
Bi-functions $L$ and $M$ on a set $X$ are rescalings if and only if $X_{L}=X_{M}$, and $L$ and $M$ are rescalings on every component of this graph. In particular, if $f,g:X\to\Cp$, then $\left(f,g\right)$ rescales $L$ to $M$ if and only if $\left(f\left|_{Y}\right.,g\left|_{Y}\right.\right)$ rescales $L$ to $M$ on every component $Y$ of $X_{L}$.
\end{proposition}
\begin{proof}
Necessity is obvious; let us prove sufficiency. Since $X_{L}=X_{M}$ we will view $X$ endowed with a fixed graph structure. Let $X=\bigsqcup\limits_{j\in I} X_{j}$, where $X_{j}$ is a graph component of $X$, for each $j\in I$. Let $f_{j},g_{j}:X_{j}\to\Cp$, and let $f,g:X\to\Cp$ be defined by $f\left(x\right)=f_{j}\left(x\right)$ and $g\left(x\right)=g_{j}\left(x\right)$, where $x\in X_{j}$.

We only need to show that if $M\left|_{X_{j}\times X_{j}}\right.=f_{j}\otimes g_{j}L\left|_{X_{j}\times X_{j}}\right.$, for every $j\in I$, then $M=f\otimes g L$. Let $x,y\in X$. If there is $j\in I$ such that $x,y\in X_{j}$, we have $f\left(x\right)g\left(y\right)L\left(x,y\right)=f_{j}\left(x\right)g_{j}\left(y\right)L\left(x,y\right)=M\left(x,y\right)$. Otherwise, $f\left(x\right)g\left(y\right)L\left(x,y\right)=0=M\left(x,y\right)$.
\end{proof}

\begin{corollary}\label{components0}Let $L$ be a bi-function on a set $X$.
\item[(i)]  $f:X\to\Cp$ reciprocally rescales $L$ to itself if and only if $f$ is constant on every component of $X_{L}$.
\item[(ii)] If $L$ is non-degenerate then for $f_{1},g_{1},f_{2},g_{2}:X\to\Cp$ we have that $f_{1}\otimes g_{1} L= f_{2}\otimes g_{2} L$ if and only if $\frac{f_{1}}{f_{2}}=\frac{g_{2}}{g_{1}}$ is constant on every component of $X_{L}$.
\end{corollary}
\begin{proof}
(i): Sufficiency follows from the preceding proposition. Let us prove necessity. Observe that if $L\left(x,y\right)\ne 0$, then from $f\left(x\right)\frac{1}{f\left(y\right)}L\left(x,y\right)=L\left(x,y\right)$ we get $f\left(x\right)=f\left(y\right)$. Hence, $f\left(x\right)=f\left(y\right)$ is an equivalence relation on $X$ that contains all pairs $\left(x,y\right)\in X\times X$ such that $L\left(x,y\right)\ne 0$. Thus, $f$ is constant on every component of $X$.\medskip

Part (ii) follows from part (i) applied to $h=\frac{f_{1}}{f_{2}}$, since if $L$ is non-degenerate, we have that $\frac{g_{1}}{g_{2}}=\frac{1}{h}$.
\end{proof}

\begin{remark}
The condition that $L$ is non-degenerate is essential in part (ii). Indeed, for $X=\left\{1,2,3,4\right\}$ consider $L:X\times X\to\C$ and $f:X\to\C$ defined by $L\left(i,j\right)=1-\left(-1\right)^{i+j}$ and $f\left(i\right)=\exp\left(\left(-1\right)^{i}\right)$, $i,j\in X$. Then $X_{L}$ is connected, but a non-constant function $f$ symmetrically rescales $L$ to itself.
\qed\end{remark}

\begin{corollary}\label{condis} Let $L$ be a bi-function on a set $X$ such that $X_{L}$ is connected. Then:
\item[(i)] $f:X\to\Cp$ reciprocally rescales $L$ to itself if and only if $f$ is constant.
\item[(ii)] Assume that $L$ is non-degenerate, let $M$ be a bi-function on $X$, and let $x\in X$. If $L$ and $M$ are rescalings, then there are unique $f,g:X\to\Cp$ such that $\left(f,g\right)$ rescales $L$ to $M$ with $f\left(x\right)=1$ (or $g\left(x\right)=1$). If $L$ and $M$ are reciprocal rescalings, then there is a unique $f:X\to\Cp$ that reciprocally rescales $L$ to $M$ with $f\left(x\right)=1$. If $L$ and $M$ are Hermitean rescalings, then there is a unique $f:X\to\Cp$ that Hermiteanly rescales $L$ to $M$ with $f\left(x\right)>0$. If $L$ and $M$ are symmetric rescalings and $f,g:X\to\Cp$ both symmetrically rescale $L$ to $M$, then either $f=g$ or $f=-g$.
\end{corollary}

The preceding results allow us to translate properties of rescalings $L$ and $M$ on properties of $f,g$ such that $\left(f,g\right)$ rescales $L$ to $M$.

\begin{proposition}\label{cresc} Let $L,M:X\times X\to\C$ be rescalings. Then:
\item[(i)] Assume that $L$ and $M$ are non-degenerate. Let $\Gamma$ be a subgroup of $\Cp$ such that for every $x,y\in X$ there is $\gamma\in\Gamma$ such that $M\left(x,y\right)=\gamma L\left(x,y\right)$. Then $L$ and $M$ are $\Gamma$-rescalings.
\item[(ii)] If $L$ and $M$ are symmetric, then they are symmetric rescalings. If moreover, they are non-degenerate $\Gamma$-rescalings, for a subgroup $\Gamma$ of $\Cp$, and $X_{L}$ is connected, then $L$ is a symmetric $\Gamma$-rescaling of  $\gamma M$, for some $\gamma\in\Gamma$.
\item[(iii)] If $L$ and $M$ are Hermitean non-degenerate and $X_{L}$ is connected, then $L$ is a Hermitean rescaling of either $M$ or $-M$.
\end{proposition}
\begin{proof}
In the light of part (i) of Proposition \ref{components}, without loss of generality we may assume that $X_{L}=X_{M}$ is connected. Let $f,g:X\to\Cp$ be such that $\left(f,g\right)$ rescales $L$ to $M$.

(i): Fix $z\in X$. From part (ii) of the preceding corollary, without loss of generality we may assume $f\left(z\right)=1$. Since $L$ and $M$ are non-degenerate, it follows that $f\left(x\right)g\left(x\right)\in\Gamma$, for every $x\in X$. If $x,y\in X$ are such that $L\left(x,y\right)\ne 0$, then $f\left(x\right)g\left(y\right)\in\Gamma$, and so $\frac{f\left(x\right)}{f\left(y\right)}=\frac{f\left(x\right)g\left(y\right)}{f\left(y\right)g\left(y\right)}\in\Gamma$. Since $\Gamma$ is a group, $\frac{f\left(x\right)}{f\left(y\right)}\in\Gamma$ is an equivalence relation. As $X_{L}$ is connected it follows that $\frac{f\left(x\right)}{f\left(y\right)}\in\Gamma$ for every $x,y\in X$. In particular, $f\left(x\right)=\frac{f\left(x\right)}{f\left(z\right)}\in\Gamma$. Finally, $g=\frac{fg}{f}$ is also $\Gamma$-valued.\medskip

(ii): Since $L$ and $M$ are symmetric, $\left(g,f\right)$ also rescales $L$ to $M$. Therefore, $e=\frac{f}{g}$ reciprocally rescales $L$ to itself, and so from part (i) of the preceding corollary, $e$ is constant. Then $h=\frac{f}{\sqrt{e}}=\sqrt{e}g$ symmetrically rescales $L$ to $M$, and so $L$ and $M$ are symmetric rescalings. If in this case $f$ and $g$ were $\Gamma$-valued, $e\equiv\gamma\in\Gamma$, and so $f$ symmetrically rescales $L$ to $f\otimes f L=\gamma f\otimes g L=\gamma M$.\medskip

(iii): Since $L$ and $M$ are Hermitean and non-degenerate, $\left(\overline{g},\overline{f}\right)$ rescales $L$ to $M$, and $fg$ is real-valued. Then $h=\frac{f}{\overline{g}}=\frac{\overline{f}}{g}$ is also real-valued. Since $h$ and $\frac{1}{h}$ rescale $L$ to itself, $h$ is a real constant. If $h>0$, then $\frac{f}{\sqrt{h}}=\overline{\sqrt{h} g}$ Hermiteanly rescales $L$ to $M$. If $h<0$, then $-\frac{f}{\sqrt{-h}}=\overline{\sqrt{-h} g}$ Hermiteanly rescales $L$ to $-M$.
\end{proof}

\begin{remark}
Note that the proof of part (i) in fact gives a stronger result: if $\left(f,g\right)$ rescales $L$ to $M$, $Y$ is a component of $X_{L}$, and $x\in Y$, then $f\left(Y\right)\subset f\left(x\right)\Gamma$, and $g\left(Y\right)\subset f\left(x\right)^{-1}\Gamma$.
\qed\end{remark}

\begin{example}\label{ecr}
Examples of $\Gamma$ that are the most relevant to the topic of rescaling are $\Rp$, $\R_{+}$, $\T$, and $\pm 1$. In particular, if $L$ and $M$ are rescalings, and $\left|L\right|=\left|M\right|$, then they are $\T$-rescalings.

If $L$ and $M$ are non-degenerate rescalings with $\frac{\widehat{M}}{\widehat{L}}>0$ (e.g. $L$ and $M$ are Hermitean rescalings), and $L^{2}=M^{2}$, then $L$ and $M$ are symmetric $\pm 1$-rescalings. Indeed, from part (i) it follows that $L$ and $M$ are $\pm 1$ rescalings, and so there are $f,g:X\to \left\{-1,1\right\}$ such that $\left(f,g\right)$ rescales $L$ to $M$. But then $fg=\frac{\widehat{M}}{\widehat{L}}>0$, and so $f=g$.\medskip
\qed\end{example}

\textbf{A criterion for rescaling. }Note that the definition of the fact that bi-functions $L$ and $M$ on a set $X$ are rescalings is extrinsic, i.e. it involves objects other than $L$ and $M$. Hence, it is desirable to be able to decide if given $L$ and $M$ are rescalings by examining a certain criterion. In this subsection we will do some preparatory work towards such a criterion.

Let $\mathcal{Y}$ stand for the class of all subsets $Y$ of $X$, such that $Y_{L}=Y_{M}$ is connected, and $L$ and $M$ are rescalings on $Y$. We will need the following technical property of this family.

\begin{lemma}\label{zorn}
If $M$ and $L$ are non-degenerate, then  $\mathcal{Y}$ satisfies the conditions of Zorn's lemma, with respect to the partial order given by inclusion of sets.
\end{lemma}
\begin{proof}
Let $I$ be a linearly directed family and let $\left\{Y_{i}\right\}_{i\in I}\subset\mathcal{Y}$ be increasing. Clearly, if $Y=\bigcup\limits_{i\in I}Y_{i}$, then $Y_{L}$ is connected. We will show that $Y\in\mathcal{Y}$.

Fix some $z\in Y$. Without loss of generality we may assume that $z\in Y_{i}$, for any $i\in I$ (otherwise restrict $I$ so that it is true; this transition does not affect $Y$). Since $L$ and $M$ are rescalings on $Y_{i}$ for every $i\in I$, there are functions $f_{i},g_{i}:Y_{i}\to\Cp$ such that $\left(f_{i},g_{i}\right)$ rescales $L$ to $M$ on $Y_{i}$ and such that $f_{i}\left(z\right)=1$.

Let $i,j\in I$ and assume that $i\prec j$. Then $Y_{i}\subset Y_{j}$, and so $\left(f_{j},g_{j}\right)$ rescales $L$ to $M$ on $Y_{i}$. Since $\left(Y_{i}\right)_{L}$ is connected, from part (ii) of Corollary \ref{condis} we have that $f_{j}\left|_{Y_{i}}\right.=f_{i}$, and $g_{j}\left|_{Y_{i}}\right.=g_{i}$. Thus, we can define functions $f,g:Y\to\Cp$, such that $f_{i}=f\left|_{Y_{i}}\right.$ and $g_{i}=g\left|_{Y_{i}}\right.$, for every $i\in I$. Let $x,y\in Y$. There is $i\in I$, such that $x,y\in Y_{i}$, and so $$M\left(x,y\right)=f_{i}\left(x\right)g_{i}\left(y\right)L\left(x,y\right)=f\left(x\right)g\left(y\right)L\left(x,y\right).$$ Since $x$ and $y$ were chosen arbitrarily, we can conclude that $Y\in\mathcal{Y}$.
\end{proof}

For bi-functions $L$ and $M$ on a set $X$ consider an equality

\begin{equation}\label{*}
M\left(x,y\right)M\left(y,z\right)L\left(x,z\right)L\left(y,y\right)=L\left(x,y\right)L\left(y,z\right)M\left(x,z\right)M\left(y,y\right),\tag*{($\ast$)}
\end{equation}

where $x,y,z\in X$. It is easy to see that if $L$ and $M$ are rescalings, then \ref{*} holds for any $x,y,z$. Conversely, in the event that there is $y\in X$ such that both $L\left(\cdot,y\right)$ and $L\left(y,\cdot\right)$ do not vanish, and \ref{*} holds for any $x,z\in X$, then $\left(f_{y},g_{y}\right)$ rescales $L$ to $M$, where $f_{y}=\frac{M\left(\cdot,y\right)L\left(y,y\right)}{L\left(\cdot,y\right)M\left(y,y\right)}$ and $g_{y}=\frac{M\left(y,\cdot\right)}{L\left(y,\cdot\right)}$.

\section{Minors of a bi-function}\label{minors}

Since a bi-function is a generalization of a square matrix, it is natural to introduce a concept related to determinants. Namely, for a bi-function $L$ on $X$ and $x_{1},...,x_{n}\in X$ denote $\det_{L}\left(x_{1},...,x_{n}\right)=\det\left[L\left(x_{i},x_{j}\right)\right]_{i,j=1}^{n}$. Any renumeration of $x_{1},...,x_{n}$ does not affect $\det_{L}\left(x_{1},...,x_{n}\right)$, since swapping $x_i$ with $x_j$ corresponds to swapping the $i$-th and $j$-th columns as well as $i$-th and $j$-th rows of the corresponding matrix, and so the determinant gets multiplied with $\left(-1\right)^{2}=1$. Analogously, if $x_{i}=x_{j}$ for some $i,j\in\overline{1,n}$, then $\det_{L}\left(x_{1},...,x_{n}\right)=0$. Hence, $\det_{L}$ may be viewed as a scalar function defined on the collection $Fin\left(X\right)$ of finite subsets of $X$. In particular, $\det_{L}\left(x\right)=\widehat{L}\left(x\right)$ and $\det_{L}\left(x,y\right)=\widehat{L}\left(x\right)\widehat{L}\left(y\right)-L\left(x,y\right)L\left(y,x\right)$, for $x,y\in X$. Also, note that $\det_{L}=\det_{L'}$ and $\det_{\overline{L}}=\det_{L^{*}}=\overline{\det_{L}}$.

Clearly, if $f:X\to\C$, then $\det_{\diag f}\left(x_{1},...,x_{n}\right)=\prod\limits_{i=1}^{n}f\left(x_{i}\right)$, for any distinct $x_{1},...,x_{n}\in X$. Also, it is easy to see that if $f,g:X\to\C$, then $\det_{f\otimes g}$ vanishes on subsets of $X$ of cardinality higher than $1$. Finally, if $L$ is symmetric (Hermitean), then $\widehat{L}\left(x\right)\widehat{L}\left(y\right)-\det_{L}\left(x,y\right)$ is equal to $L\left(x,y\right)^{2}$ ($\left|L\left(x,y\right)\right|^{2}$), for $x,y\in X$.

It is natural to ask how certain relations between two bi-functions $L$ and $M$ reflect on the relations between $\det_{L}$ and $\det_{M}$. First, let us consider the case when $L$ and $M$ are rescalings.

\begin{proposition}\label{pdc}
Let $L$ and $M$ be bi-functions on $X$, and let $f,g:X\to\C$ be such that $M=f\otimes g L$. Then, $\det_{M}=\det_{\diag fg}\det_{L}$, i.e. for any $x_{1},...,x_{n}\in X$ we have
$$\det\nolimits_{M}\left(x_{1},...,x_{n}\right)=\prod\limits_{i=1}^{n}f\left(x_{i}\right)g\left(x_{i}\right)\det\nolimits_{L}\left(x_{1},...,x_{n}\right).$$
\end{proposition}
\begin{proof}
The matrix $\left[M\left(x_{i},x_{j}\right)\right]_{i,j=1}^{n}$ is obtained from the matrix $\left[L\left(x_{i},x_{j}\right)\right]_{i,j=1}^{n}$ by multiplying $i$-th row with $f\left(x_{i}\right)$ and $i$-th column with $g\left(x_{i}\right)$, for every $i\in\overline{1,n}$. Hence, the identity follows from the properties of determinants.
\end{proof}

On the other hand, it would be interesting to see what are the relations between $L$ and $M$ for which $\det_{L}=\det_{M}$. We immediately get $\widehat{L}=\widehat{M}$, and so $L$ and $M$ coincide on the diagonal. It is also easy to see that if $L$ and $M$ are symmetric (or Hermitean) then $\det_{L}=\det_{M}$ on sets of cardinality $1$ and $2$ if and only if $L^{2}=M^{2}$, (or $\left|L\right|=\left|M\right|$), and $\widehat{L}=\widehat{M}$. In particular, in this case $X_{L}=X_{M}$. Slightly less obvious property is given in the following proposition.

\begin{proposition}\label{cal}
Let $L$ and $M$ be bi-functions on $X$. Then if $\det_{L}=\det_{M}$, then $\det_{L+\diag h}=\det_{M+\diag h}$, for any $h:X\to\C$.
\end{proposition}

In order to prove the proposition, we will need the following lemma.

\begin{lemma}
Let $M$ be a $n\times n$ complex matrix, let $\mu=\left(\mu_{1},...,\mu_{n}\right)\in\C^{n}$ be a row and let $\nu=\left(\nu_{1},...,\nu_{n}\right)\in\C^{n}$ be a column. Then for any $a,b\in\C$ we have $$\det\left[\begin{array}{rr} a & \mu \\ \nu & M \end{array}\right]-\det\left[\begin{array}{rr} b & \mu \\ \nu & M \end{array}\right]=\left(a-b\right)\det M.$$
\end{lemma}
\begin{proof}
Using the fact that determinant is a polylinear functional we get $$\det\left[\begin{array}{rr} a & \mu \\ \nu & M \end{array}\right]-\det\left[\begin{array}{rr} b & \mu \\ \nu & M \end{array}\right]=\det\left[\begin{array}{rr} a-b & 0_{\C^{n}} \\ \nu & M \end{array}\right]=\left(a-b\right)\det M,$$ where the last equalty follows from the Laplace expansion over the first row.
\end{proof}

\begin{proof}[Proof of Proposition \ref{cal}]
Let $L$ and $M$ be bi-functions on $X$ such that $\det_{L}=\det_{M}$. Let us start with showing that for any $x\in X$ and $\alpha\in\C$ we have $\det_{L_{1}}=\det_{M_{1}}$, where $L_{1}=L+\alpha\diag \delta_{x}$ and $M_{1}=M+\alpha\diag \delta_{x}$. Let $x_{1},...,x_{n}\in X$ be distinct. If this collection of points does not contain $x$, then $\left.L_{1}\right|_{\left\{x_{1},...,x_{n}\right\}}=\left.L\right|_{\left\{x_{1},...,x_{n}\right\}}$ and $\left.M_{1}\right|_{\left\{x_{1},...,x_{n}\right\}}=\left.M\right|_{\left\{x_{1},...,x_{n}\right\}}$. Hence,
$$\det\nolimits_{L_{1}}\left(x_{1},...,x_{n}\right)=\det\nolimits_{L}\left(x_{1},...,x_{n}\right)=\det\nolimits_{M}\left(x_{1},...,x_{n}\right)=\det\nolimits_{M_{1}}\left(x_{1},...,x_{n}\right).$$
If $x$ is present among $x_{1},...,x_{n}$, without loss of generality we may assume that $x_{1}=x$. In this case it follows from the lemma that
\begin{align*}
\det\nolimits_{L_{1}}\left(x_{1},...,x_{n}\right)-\det\nolimits_{L}\left(x_{1},...,x_{n}\right)&=\alpha \det\nolimits_{L}\left(x_{2},...,x_{n}\right)\\=\alpha \det\nolimits_{M}\left(x_{2},...,x_{n}\right)&=\det\nolimits_{M_{1}}\left(x_{1},...,x_{n}\right)-\det\nolimits_{M}\left(x_{1},...,x_{n}\right),
\end{align*}
from where $\det_{L_{1}}\left(x_{1},...,x_{n}\right)=\det_{M_{1}}\left(x_{1},...,x_{n}\right)$.\medskip

Now, let us show that for any distinct $x_{1},...,x_{n}\in X$ we have $$\det\nolimits_{L+\diag h}\left(x_{1},...,x_{n}\right)=\det\nolimits_{M+\diag h}\left(x_{1},...,x_{n}\right).$$ Recursively, define a sequence $\left\{L_{i}\right\}_{i=0}^{n}$ of bi-functions on $X$ by $L_{0}=L$ and $L_{i}=L_{i-1}+h\left(x_{i}\right)\delta_{x_{i}}$, for every $i\in\overline{1,n}$. Define $\left\{M_{i}\right\}_{i=0}^{n}$ analogously. Using induction and the previous step we get that $\det_{L_{i}}=\det_{M_{i}}$, for every $i\in\overline{1,n}$. Also, note that $\left.L_{n}\right|_{\left\{x_{1},...,x_{n}\right\}}=\left.\left[L+\diag h\right]\right|_{\left\{x_{1},...,x_{n}\right\}}$ and $\left.M_{n}\right|_{\left\{x_{1},...,x_{n}\right\}}=\left.\left[M+\diag h\right]\right|_{\left\{x_{1},...,x_{n}\right\}}$, and so $$\det\nolimits_{L+\diag h}\left(x_{1},...,x_{n}\right)=\det\nolimits_{L_{n}}\left(x_{1},...,x_{n}\right)=\det\nolimits_{M_{n}}\left(x_{1},...,x_{n}\right)=\det\nolimits_{M+\diag h}\left(x_{1},...,x_{n}\right).$$ Since $x_{1},...,x_{n}$ were chosen arbitrarily, we conclude that $\det_{L+\diag h}=\det_{M+\diag h}$.
\end{proof}

From Proposition \ref{pdc} it follows that if $M$ to $L$ are reciprocal rescalings, then $\det_{L}=\det_{M}$. It turns out that in the class of symmetric bi-functions this assertion can be reversed.

\begin{theorem}\label{similar}
Symmetric bi-functions $L$ and $M$ on a set $X$ are $\pm1$ symmetric rescalings if and only if $\det_{L}=\det_{M}$ (as functions on $Fin\left(X\right)$).
\end{theorem}

This result for the case when $X$ is finite was first proven in \cite{es} and then in \cite{oeding} in relation to Principal Minor Assignment Problem. In \cite{rkt} the algorithmical side of the problem was considered. Below we adapt combinatorial proof from \cite{es} and \cite{rkt} to the infinite case. However, in Section \ref{geoin} we will present a geometric interpretation of the theorem and give an alternative geometric proof under some restrictions. Note that since $\det_{L}=\det_{L'}$ the result is specific to the case of symmetric bi-functions, although under some additional assumptions $\det_{L}=\det_{M}$ implies that either $L$ and $M$, or $L'$ and $M$ are $\pm1$ symmetric rescalings (see e.g. \cite{bc} which discusses the skew-symmetric case, and also \cite{loewy}). An example of Hermitean bi-functions $L$ and $M$ on $\left\{1,2,3,4\right\}$ such that $\det_{L}=\det_{M}$, but neither $M$ nor $M'$ is a rescaling of $L$ are given by the matrices

$$\left[\begin{array}{rrrr} 4 & e^{\ii\frac{\pi}{12}} & 1 & 1 \\ e^{-\ii\frac{\pi}{12}} & 4& 1 & e^{\ii\frac{\pi}{4}} \\ 1 & 1 & 4 & e^{\ii\frac{\pi}{3}} \\ 1 & e^{-\ii\frac{\pi}{4}} & e^{-\ii\frac{\pi}{3}} & 4 \end{array}\right] \mbox{ and } \left[\begin{array}{rrrr} 4 & e^{\ii\frac{\pi}{6}} & e^{\ii\frac{\pi}{12}} & 1 \\ e^{-\ii\frac{\pi}{6}} & 4& 1 & e^{\ii\frac{\pi}{6}} \\ e^{-\ii\frac{\pi}{12}} & 1 & 4 & e^{\ii\frac{\pi}{12}} \\ 1 & e^{-\ii\frac{\pi}{6}} & e^{-\ii\frac{\pi}{12}} & 4 \end{array}\right]. $$
\

\textbf{Proof of the theorem. }In order to prove the theorem we will need the following technical lemma.

\begin{lemma}\label{matrix}
Let $n\ge 3$ and let $a_{1},...,a_{n}$, $b_{1},...,b_{n}$ and $c_{1},...,c_{n}$ be complex numbers, such that $b_{i}=\pm c_{i}$, for every $i\in\overline{1,n}$. Then $\prod\limits_{i=1}^{n}b_{i}=\prod\limits_{i=1}^{n}c_{i}$ if and only if
$$\det \left[\begin{array}{rrrrrr} a_{1} & b_{1} & 0 & \cdots & 0 & b_{n} \\ b_{1} & a_{2} & b_{2} & \ddots &  & 0 \\ 0 & b_{2} & \ddots & \ddots & \ddots & \vdots \\ \vdots & \ddots & \ddots & \ddots & b_{n-2} & 0 \\ 0 &  & \ddots & b_{n-2} & a_{n-1} & b_{n-1} \\ b_{n} & 0 & \cdots & 0 & b_{n-1} & a_{n} \end{array}\right]=\det \left[\begin{array}{rrrrrr} a_{1} & c_{1} & 0 & \cdots & 0 & c_{n} \\ c_{1} & a_{2} & c_{2} & \ddots &  & 0 \\ 0 & c_{2} & \ddots & \ddots & \ddots & \vdots \\ \vdots & \ddots & \ddots & \ddots & c_{n-2} & 0 \\ 0 &  & \ddots & c_{n-2} & a_{n-1} & c_{n-1} \\ c_{n} & 0 & \cdots & 0 & c_{n-1} & a_{n} \end{array}\right].$$
\end{lemma}
\begin{proof}
Denote the determinants in the statement by $d_{n}\left(a_{1},...,a_{n},b_{1},...,b_{n}\right)$ and \linebreak $d_{n}\left(a_{1},...,a_{n},c_{1},...,c_{n}\right)$. For $a_{1},a_{2},b\in \C$ define $d_{2}\left(a_{1},a_{2},b,0\right)=a_{1}a_{2}-b^{2}$. It is clear that $d_{2}\left(a_{1},a_{2},b,0\right)=d_{2}\left(a_{1},a_{2},-b,0\right)$.

The proof is done by induction. Here we only provide a sketch. The case when $n=3$ is a simple computation.  For $n>3$ expanding the determinant by the first row, and then by the first columns of the (second and third) obtained matrices we get the recursive formula
\begin{align*}
 d_{n}\left(a_{1},...,a_{n},b_{1},...,b_{n}\right)+2\left(-1\right)^{n}\prod\limits_{i=1}^{n}b_{i}=a_{1}d_{n-1}\left(a_{2},...,a_{n},b_{2},...,b_{n-1},0\right)\\
-b_{1}^{2}d_{n-2}\left(a_{3},...,a_{n},b_{3},...,b_{n-1},0\right)-b_{n}^{2}d_{n-2}\left(a_{2},...,a_{n-1},b_{2},...,b_{n-2},0\right).
\end{align*}
Using this formula and the hypothesis of induction the result follows.
\end{proof}

Recall that an \emph{induced cycle} in a graph is a cycle such that the vertices are adjacent in the cycle if and only if they are adjacent in the original graph. Theorem \ref{similar} follows immediately from the following result, essentially proven in \cite{es}.

\begin{proposition}Let $L$ be a symmetric bi-function on a set $X$ such that the lengths of all the induced cycles in $X_{L}$ are less than $l\in\N\backslash\left\{1,2\right\}\cup\left\{\8\right\}$. Then any symmetric bi-function $M$ on $X$ is a $\pm1$ symmetric rescaling of $L$ whenever $\det_{L}=\det_{M}$ on all of the subsets of $X$ of cardinality less than $l$.
\end{proposition}
\begin{proof}
Since $L$ and $M$ are symmetric with $\det_{L}=\det_{M}$ on all of the subsets of $X$ of cardinality less than $l>2$, it follows that $\widehat{L}=\widehat{M}$ and $L^{2}=M^{2}$. In the light of Proposition \ref{components} we can assume that $X_{L}=X_{M}$ is connected. Let us start with showing that if $x_{0},x_{1},...,x_{n}=x_{0}$, $n>2$, is a cycle in $X_{L}$, then $\prod\limits_{i=0}^{n-1}\frac{M\left(x_{i},x_{i+1}\right)}{L\left(x_{i},x_{i+1}\right)}=1$. In the case when it is an induced cycle we have $n<l$ and the claim follows from the Lemma \ref{matrix} applied to the numbers $a_{i}=L\left(x_{i},x_{i}\right)=M\left(x_{i},x_{i}\right)$, $b_{i}=L\left(x_{i},x_{i+1}\right)$ and $c_{i}=M\left(x_{i},x_{i+1}\right)$, $i\in \overline{0,n-1}$.

Now argue by induction. For $n=3$ any cycle is induced, and so the equality holds. Assume that the claim is true for all $k\in\overline{3,n}$ and let $x_{0},x_{1},...,x_{n},x_{n+1}=x_{0}$ be a cycle. Without loss of generality, we may assume that all $x_{0},x_{1},...,x_{n}$ are all different (otherwise by the hypothesis of induction, the product becomes the product of several copies of $1$). Since we only need to consider non-induced cycles, without loss of generality we may assume that there is $k\in\overline{2,n-1}$ such that $x_{0}$ and $x_{k}$ are joined with an edge. Then, from the hypothesis of induction applied to the cycles $x_{0},x_{1},...,x_{k},x_{0}$ and $x_{0},x_{k},x_{k+1},...,x_{n},x_{n+1}=x_{0}$ we have that $$\frac{M\left(x_{0},x_{k}\right)^{2}}{L\left(x_{0},x_{k}\right)^{2}}\prod\limits_{i=0}^{n}\frac{M\left(x_{i},x_{i+1}\right)}{L\left(x_{i},x_{i+1}\right)}=1.$$ Hence, the claim follows since $L\left(x_{0},x_{k}\right)^{2}=M\left(x_{0},x_{k}\right)^{2}\ne 0$.\medskip

Fix $z\in X$. Define $f:X\to\left\{-1,1\right\}$ by $f\left(x\right)=\prod\limits_{i=0}^{n-1}\frac{M\left(x_{i},x_{i+1}\right)}{L\left(x_{i},x_{i+1}\right)}$, where $z=x_{0},x_{1},...,x_{n}=x$ is a path in $X_{L}$ from $z$ to $x$ (this product is always equal to $\pm 1$ as $L^{2}=M^{2}$). The function is well-defined, since if $z=x_{0},x_{1},...,x_{n}=x$ and $z=y_{0},y_{1},...,y_{m}=x$ are two paths, then $z=x_{0},x_{1},...,x_{n}=x=y_{m},...,y_{0}=z$ is a cycle, and so from the claim above $$\prod\limits_{i=0}^{n-1}\frac{M\left(x_{i},x_{i+1}\right)}{L\left(x_{i},x_{i+1}\right)}=\left(\prod\limits_{i=0}^{m-1}\frac{M\left(y_{i},y_{i+1}\right)}{L\left(y_{i},y_{i+1}\right)}\right)^{-1}=\prod\limits_{i=0}^{m-1}\frac{M\left(y_{i},y_{i+1}\right)}{L\left(y_{i},y_{i+1}\right)}.$$

Now let us show that $M=f\otimes f L$. If $x,y\in X$ are not joined by an edge in $X_{L}$, then $M\left(x,y\right)=0=f\left(x\right)f\left(y\right)L\left(x,y\right)$. Assume that $x,y\in X$ are joined by an edge. Let $z=x_{0},x_{1},...,x_{n}=x$ be a path from $z$ to $x$. Then, $z=x_{0},x_{1},...,x_{n}=x,x_{n+1}=y$ is a path from $z$ to $y$. Hence, $f\left(y\right)=\frac{M\left(x,y\right)}{L\left(x,y\right)}f\left(x\right)$, and since $f\left(x\right)=\pm 1=\frac{1}{f\left(x\right)}$ we conclude that $M\left(x,y\right)=f\left(x\right)f\left(y\right)L\left(x,y\right)$.
\end{proof}

One can ask what other characteristics of the graph $X_{L}$ can be used to reduce the size of sets on which $\det_{L}=\det_{M}$ has to be tested in order to conclude that symmetric $L$ and $M$ are $\pm1$ symmetric rescalings. Recall that the \emph{(closed) neighborhood} of a vertex in a graph is the set of all adjacent vertices (together with the original vertex). More generally, the (closed) $n$-th neighborhood of a vertex in a graph is the set of all vertices of distance $1,2,...,n$ ($0,1,...,n$). The component that contains a vertex is the union of all of its closed neighborhoods. \emph{Radius} of a graph is the minimal $n\in\N$ such that there is a vertex $z$ such that the graph is equal to the closed $n$-th neighborhood of $z$.

\begin{proposition}\label{radius}Let $L$ be a symmetric bi-function on a set $X$ such that the radius of every component of $X_{L}$ does not exceed $l\in\N$. Then any symmetric bi-function $M$ on $X$ is a $\pm1$ symmetric rescaling of $L$ whenever $\det_{L}=\det_{M}$ on all of the subsets of $X$ of cardinality at most $2l+1$.
\end{proposition}
\begin{proof}
We will provide a sketch of the proof. Again, we only need to consider the case when $X_{L}$ is connected.

In the same way as above one can show that if $x_{0},x_{1},...,x_{n}=x_{0}$, $2l+1\ge n>2$, is a cycle in $X$, then $\prod\limits_{i=0}^{n-1}\frac{M\left(x_{i},x_{i+1}\right)}{L\left(x_{i},x_{i+1}\right)}=1$. Fix $z\in X$ such that $X$ is the closed $l$-th neighborhood of $z$, and define $f:X\to\left\{-1,1\right\}$ by $f\left(x\right)=\prod\limits_{i=0}^{n-1}\frac{M\left(x_{i},x_{i+1}\right)}{L\left(x_{i},x_{i+1}\right)}$, where $z=x_{0},x_{1},...,x_{n}=x$ is a shortest path in $X_{L}$ from $z$ to $x$ (and so $n\le l$). Analogously to the previous proof we can show that $f$ is well-defined.

For $x,y\in X$ we need to show that if $\left(x,y\right)$ is an edge in $X_{L}$, then $M\left(x,y\right)=f\left(x\right)f\left(y\right)L\left(x,y\right)$. Let $z=x_{0},x_{1},...,x_{n}=x$ and $z=y_{0},y_{1},...,y_{m}=y$ be shortest paths from $z$ to $x$ and $y$, respectively. Then using the fact that $z=x_{0},x_{1},...,x_{n}=x,y=y_{m},...,y_{1},y_{0}=z$ is a cycle of length $m+n+1\le 2l+1$, from the claim above we can deduce the required equality.
\end{proof}

\begin{remark}
It is possible for a graph to have an infinite radius but contain only small induced cycles. Indeed, the graph of a bi-function $L$ from Example \ref{ex} below has an infinite radius and no induced cycles. It is also possible for a graph of a small radius to contain arbitrarily large induced cycles. Indeed, consider a disconnected union of the cycles of all possible length (starting with $3$), choose a vertex in the first cycle and join that vertex with every other vertex of the union. The obtained graph contains induced cycles of all lengths, but its radius is $1$.
\qed\end{remark}

In the case when there is $y\in X$ such that $L\left(\cdot,y\right)$ and $L\left(y,\cdot\right)$ do not vanish, the radius of $X_{L}$ is $1$, which gives us the following criterion.

\begin{corollary}\label{nv} Let $L$ and $M$ be non-degenerate symmetric bi-functions on $X$. Then the following are equivalent:
\item[(i)] $\frac{\widehat{M}}{\widehat{L}}>0$, and for all $x,y,z\in X$ we have
 $$L\left(x,y\right)L\left(y,z\right)L\left(z,x\right)=M\left(x,y\right)M\left(y,z\right)M\left(z,x\right);$$
\item[(ii)] $L^{2}=M^{2}$, $\widehat{L}=\widehat{M}$, and \ref{*} holds for all $x,y,z\in X$;
\item[(iii)] $\det_{L}=\det_{M}$ on sets of cardinality $1$, $2$ and $3$.

If moreover there is $y\in X$ such that $L\left(\cdot,y\right)$ does not vanish, then the conditions (i)-(iii) are equivalent to the fact that $L$ and $M$ are symmetric $\pm1$ rescalings.
\end{corollary}
\begin{proof}
First, observe that plugging $x=y=z$ in the equality in (i), we get that $L\left(x,x\right)^{3}=M\left(x,x\right)^{3}$. Since we also have $\frac{\widehat{M}}{\widehat{L}}>0$, it follows that $\widehat{L}=\widehat{M}$. Plugging $x=z$ gives us $L\left(x,y\right)^{2}L\left(x,x\right)=M\left(x,y\right)^{2}M\left(x,x\right)$, from where $L\left(x,y\right)^{2}=M\left(x,y\right)^{2}$. Now both  \ref{*} and the formula in (i) hold if $L\left(z,x\right)=0$; otherwise the latter is transformed into the former by multiplying with the equality $\frac{M\left(z,x\right)M\left(y,y\right)}{L\left(z,x\right)}=\frac{L\left(z,x\right)L\left(y,y\right)}{M\left(z,x\right)}$. Finally, a simple calculation shows that under the assumption that $\widehat{L}=\widehat{M}$ and $L^{2}=M^{2}$, the formula in (i) is equivalent to $\det_{L}\left(x,y,z\right)=\det_{M}\left(x,y,z\right)$.
\end{proof}

The following example shows that in general we cannot conclude that $L$ and $M$ are $\pm1$ symmetric rescalings when $\det_{L}=\det_{M}$ on small sets. In particular, if $X$ is infinite, in order to decide whether $L$ and $M$ are $\pm1$ symmetric rescalings we have to guarantee $\det_{L}=\det_{M}$ on sets of arbitrary size.

\begin{example}\label{ex} For every $n>2$ consider the $n\times n$ matrices $$L_{n}^{+}=\left[\begin{array}{rrrrrr} 4 & 1 & 0 & \cdots & 0 & 1 \\ 1 & 4 & 1 & \ddots &  & 0 \\ 0 & 1 & \ddots & \ddots & \ddots & \vdots \\ \vdots & \ddots & \ddots & \ddots & 1 & 0 \\ 0 &  & \ddots & 1& 4 & 1 \\ 1 & 0 & \cdots & 0 & 1 & 4 \end{array}\right]\mbox{ and }L_{n}^{-}=\left[\begin{array}{rrrrrr} 4 & 1 & 0 & \cdots & 0 & -1 \\ 1 & 4 & 1 & \ddots &  & 0 \\ 0 & 1 & \ddots & \ddots & \ddots & \vdots \\ \vdots & \ddots & \ddots & \ddots & 1 & 0 \\ 0 &  & \ddots & 1& 4 & 1 \\ -1 & 0 & \cdots & 0 & 1 & 4 \end{array}\right].$$
Applying Lemma \ref{matrix} one can show that all corresponding proper principal minors of these matrices coincide, but $\det L_{n}^{+}\ne \det L_{n}^{-}$, and so $L_{n}^{+}$ and $L_{n}^{-}$ are not $\pm1$ symmetric rescalings. Hence, from Example \ref{ecr}, $L_{n}^{+}$ and $L_{n}^{-}$ are not rescalings at all. In particular, $L_{4}^{+}$ and $L_{4}^{-}$ are not rescalings despite satisfying conditions (i)-(iii) of Corollary \ref{nv}.\medskip

Consider a ``basal'' bi-function $L$ on $X=\N$ defined by $$L\left(m,n\right)=\left\{\begin{array}{lll} 4 & m=n \\ 1  & \left|m-n\right|=1 \\ 0  & \left|m-n\right|>1
 \end{array}\right.$$
The graph $X_{L}$ is the infinite path $1,2,3,...$. By adding extra edges to this graph we will increase its complexity. Namely, for any $A\subset \left\{2,3,4,...\right\}$ define $L_{A}$ by
$$L_{A}\left(m,n\right)=\left\{\begin{array}{lll} L\left(m,n\right)+1 & m+n=\left(m-n\right)^{2}\mbox{ and }\left|m-n\right|\in A \\ L\left(m,n\right)-1 & m+n=\left(m-n\right)^{2}\mbox{ and }\left|m-n\right|\in \N\backslash A \\ L\left(m,n\right) & \mbox{otherwise}
 \end{array}\right.$$
Note that $m+n=\left(m-n\right)^{2}$ means that $m=\frac{\left(m-n\right)^{2}+\left(m-n\right)}{2}$ and $n=\frac{\left(m-n\right)^{2}-\left(m-n\right)}{2}$. Therefore, the graph $X_{L_{A}}$ consists of the infinite path $1,2,3,...$ with additional edges $\left(\frac{k^{2}-k}{2},\frac{k^{2}+k}{2}\right)$, for $k>1$. Hence, for every $k\in A$ we have $\left[L_{A}\left(m,n\right)\right]_{m,n=\frac{k^{2}-k}{2}}^{\frac{k^{2}+k}{2}}=L_{k}^{+}$ and for $k>1$, $k\not\in A$ we have $\left[L_{A}\left(m,n\right)\right]_{m,n=\frac{k^{2}-k}{2}}^{\frac{k^{2}+k}{2}}=L_{k}^{-}$. Thus, $L_{A}$ and $L_{B}$ are $\pm1$ symmetric rescalings, for $A,B\subset \left\{2,3,4,...\right\}$ if and only if $A=B$, and in order to guarantee that we have to check that $\det_{L_{A}}=\det_{L_{B}}$ on sets of arbitrary size.\medskip\qed
\end{example}

Assume that $L$ is a bi-function on $X$, such that $\det_{L}$ does not vanish. Then, if $M$ is a rescaling of $L$, it follows from Proposition \ref{pdc} that $\det_{M}$ does not vanish and $\det_{M}\det_{L}^{-1}$ is a multiplicative function on $Fin\left(X\right)$. In fact, the converse is also true if we assume that $L$ and $M$ are symmetric.

\begin{proposition}\label{pdcr}
Let $L$ and $M$ be symmetric bi-functions on $X$. Assume that $\det_{L}$ does not vanish and $M$ is non-degenerate. Then $M$ is a rescaling of $L$ if and only if $\det_{M}\det_{L}^{-1}$ is a multiplicative function on $Fin\left(X\right)$.
\end{proposition}
\begin{proof}
We only need to prove sufficiency. Assume that $\det_{M}\det_{L}^{-1}$ is a multiplicative function on $Fin\left(X\right)$, and so there is $f:X\to \C$ such that for any $x_{1},...,x_{n}\in X$ we have
$\det_{M}\left(x_{1},...,x_{n}\right)=\prod\limits_{i=1}^{n}f\left(x_{i}\right)\det_{L}\left(x_{1},...,x_{n}\right)$. Since $M$ is non-degenerate, it follows that $f$ does not vanish. Let $g:X\to \Cp$ be such that $g^{2}=f$. Then $K=g\otimes g L$ is a symmetric bi-function such that $\det_{M}=\det_{K}$. From Theorem \ref{similar} there is $h:X\to \left\{-1,1\right\}$ such that $M=h\otimes h K$. Hence, $M=gh\otimes gh L$ is a rescaling of $L$.
\end{proof}

\section{Rescaling on topological spaces}\label{restop}

Until now we studied general bi-functions on sets without any additional structure. In this section we will consider the topological aspect of the topic. Let us start with the following observation: if $L$ and $M$ are separately continuous non-vanishing bi-functions with $L^{2}=M^{2}$, then either $L=M$ or $L=-M$. Indeed, for every $y$ we have $\frac{L\left(\cdot,y\right)}{M\left(\cdot,y\right)}=\pm 1$, and from the continuity we get that either $L\left(\cdot,y\right)=M\left(\cdot,y\right)$, or $L\left(\cdot,y\right)=-M\left(\cdot,y\right)$. Analogously, either $L\left(x,\cdot\right)=M\left(x,\cdot\right)$, or $L\left(x,\cdot\right)=-M\left(x,\cdot\right)$, for every $x\in X$. Combining these assertions we the claim follows.

It is clear that we heavily relied on the assumption that $L$ and $M$ do not vanish, and so the properties of the graph $X_{K}=X_{L}$ come into play. In this section we will investigate the connection between the topology of $X$ and the graph structure of $X_{L}$ when $L$ satisfies certain minimal assumptions of continuity.\medskip

\textbf{A graph of a continuous bi-function. }Let $L$ be a bi-function on a set $X$. For $y\in X$ define $U_{y}=\left\{x\in X\left|L\left(x,y\right)\ne 0\right.\right\}$ and $U^{y}=\left\{x\in X\left|L\left(y,x\right)\ne 0\right.\right\}$. Note that $x\in U_{y}\Leftrightarrow y\in U^{x}$ and $L\left(y,y\right)\ne 0\Leftrightarrow y\in U_{y}\Leftrightarrow y\in U^{y}$. In fact, $U_{y}\cup U^{y}$ is either the neighborhood, or the closed neighborhood of $y$ in $X_{L}$, depending on whether $L\left(y,y\right)= 0$.

Assume that $M$ and $L$ are bi-functions on $X$, and $y\in X$ and $y\in Y\subset U_{y}$ are such that there are $f,g:X\to\Cp$ are such that $f\left(x\right)g\left(z\right)L\left(x,z\right)=M\left(x,z\right)$, for every $x,z\in Y$. Then on $Y$ we have $f=\frac{M\left(\cdot,y\right)}{g\left(y\right) L\left(\cdot,y\right)}$, and so $f$ is determined uniquely on $Y$ up to a constant multiple. Analogously, $g$ is determined uniquely on $Y$ up to a constant multiple in the case when $y\in Y\subset U^{y}$.

Now assume that $X$ is a topological space and $x,y\in X$ are such that $L\left(x,y\right)\ne 0$. If $L$ is continuous in the first variable at $\left(x,y\right)$, then $x\in\Int U_{y}$. Every point of $U_{y}$ is of distance at most $2$ from $x$ in $X_{L}$, and so the closed second neighborhood of $x$ in $X_{L}$ is a (not necessarily open) neighborhood of $x$ in the topological space $X$. In particular, if in this case $x=y$, then $y\in\Int U_{y}$, and so the closed neighborhood of $y$ in $X_{L}$ is also a topological neighborhood. This observation leads to the following result.

\begin{proposition}\label{ocom}
Let $X$ be a topological space and let $L$ be a bi-function on $X$. Then the components of $X_{L}$ are open once $L$ satisfies one of the following properties:
\item[(i)] $L$ is non-degenerate and separately continuous in the first variable at the points of the diagonal;
\item[(i')] $L$ is non-degenerate and separately continuous in the second variable at the points of the diagonal;
\item[(ii)] $L$ is separately continuous in the first variable and for every $x\in X$ there is $y\in X$ such that $L\left(x,y\right)\ne 0$.
\item[(ii')] $L$ is separately continuous in the second variable and for every $x\in X$ there is $y\in X$ such that $L\left(y,x\right)\ne 0$.
\item[(iii)] $L$ is separately continuous and the graph $X_{L}$ has no isolated vertices.
\end{proposition}

Another consequence of separate continuity of bi-functions is the continuity of the functions that rescale them.

\begin{proposition}\label{con}
Let $X$ be a topological space and let $L$ and $M$ be bi-functions on $X$, which satisfy the same of the conditions (i), (ii) or (iii) of the preceding proposition. If $f,g:X\to\Cp$ are such that $\left(f,g\right)$ rescales $L$ to $M$, then $f$ is continuous.
\end{proposition}
\begin{proof}
We will only provide a proof for the case when both $L$ and $M$ satisfy the condition (ii) of the preceding proposition, since the other proofs are similar. Fix $x\in X$. There is $y\in X$ such that $L\left(x,y\right)\ne 0$, and so $x\in \Int U_{y}$. Recall that on $U_{y}$ we have $f=\frac{\lambda M\left(\cdot,y\right)}{L\left(\cdot,y\right)}$, for some $\lambda$ that depends on $y$. Since both $M\left(\cdot,y\right)$ and $L\left(\cdot,y\right)$ are continuous, it follows that $f$ is continuous at $x$. Since $x$ was chosen arbitrarily we conclude that $f$ is continuous on $X$.
\end{proof}
\begin{remark}
If in the notations of the proposition $L$ and $M$ satisfy the same of the conditions (i'), (ii') or (iii) of Proposition \ref{ocom}, then $g$ is continuous.\qed
\end{remark}

We can now show that if two bi-functions are ``almost rescalings'', then they are in fact rescalings.

\begin{proposition}\label{conde}
Let $X$ be a topological space and let $L$ and $M$ be separately continuous non-degenerate bi-functions on $X$. Assume that $L$ and $M$ are rescalings on a dense set $Y\subset X$, and also for any $x\in X$ there is a neighborhood $U$ of $x$ such that $L$ and $M$ are rescalings on $U$. Then $L$ and $M$ are rescalings on $X$.
\end{proposition}
\begin{proof}
First, note that we do not have $X_{L}=X_{M}$ yet, and so $U_{x}$ and $U^{x}$, for $x\in X$, might be different with respect to $L$ and $M$. However, since these bi-functions are separately continuous and non-degenerates, these sets are open neighborhoods of $x$. Let $V_{x}$ be the intersection of all these four sets ($U_{x}$ and $U^{x}$ with respect to $L$, and $U_{x}$ and $U^{x}$ with respect to $M$), which is an open neighborhood of $x$.\medskip

Let $f,g:X\to\Cp$ be such that $\left(f,g\right)$ rescales $L$ to $M$ on $Y$. Let $x\in X$ and let $U$ be an open neighborhood of $x$ such that $L$ and $M$ are rescalings on $U$. Then $V_{x}\cap U$ is an open neighborhood of $x$, and since $Y$ is dense, there is $y\in Y\cap V_{x}\cap U$. Since $y\in V_{x}$ it follows that $x\in V_{y}$, and so $W_{x}=V_{x}\cap V_{y}\cap U$ is an open neighborhood of both $x$ and $y$. Let $f_{x},g_{x}:U\to\Cp$ be such that $\left(f_{x},g_{x}\right)$ rescales $L$ to $M$ on $U$ with $f_{x}\left(y\right)=f\left(y\right)$. Then both $\left(f,g\right)$ and $\left(f_{x},g_{x}\right)$ rescale $L$ to $M$ on $W_{x}\cap Y$ with $f_{x}\left(y\right)=f\left(y\right)$, and $y\in W_{x}\cap Y\subset V_{y}$. Hence, $\left.f_{x}\right|_{W_{x}\cap Y}=\left.f\right|_{W_{x}\cap Y}$, and $\left.g_{x}\right|_{W_{x}\cap Y}=\left.g\right|_{W_{x}\cap Y}$.\medskip

Let $x,z\in X$ and let $W_{x}$ and $W_{z}$ be as constructed above. Then $\left.f_{x}\right|_{W_{x}\cap Y\cap W_{z}}=\left.f\right|_{W_{x}\cap Y\cap W_{z}}=\left.f_{z}\right|_{W_{x}\cap Y\cap W_{z}}$, and both $f_{x}$ and $f_{z}$ are continuous due to Proposition \ref{con}. Since $W_{x}\cap Y\cap W_{z}$ is dense in $W_{x}\cap W_{z}$, we conclude that $\left.f_{x}\right|_{W_{x}\cap W_{z}}=\left.f_{z}\right|_{W_{x}\cap W_{z}}$.

Thus, the extensions of $f$ and $g$ on $X$ given by $f\left(x\right)=f_{x}\left(x\right)$ and $g\left(x\right)=g_{x}\left(x\right)$, for $x\in X$, are well defined and continuous. It is left to show that $\left(f,g\right)$ rescales $L$ to $M$ on $X$. We know that $f\left(x\right)g\left(y\right)L\left(x,y\right)=M\left(x,y\right)$, whenever $x,y\in Y$. Fix $y\in Y$. Since $L\left(\cdot,y\right)$, $M\left(\cdot,y\right)$ and $f$ are continuous, and $Y$ is dense, it follows that $f\left(x\right)g\left(y\right)L\left(x,y\right)=M\left(x,y\right)$, for every $x\in X$. Applying the same argument to the second variable, we conclude that $\left(f,g\right)$ rescales $L$ to $M$ on $X$.
\end{proof}
\begin{remark}
If in the statement of the proposition $L$ and $M$ were symmetric / Hermitean/ reciprocal rescalings on $Y$, then they are symmetric / Hermitean/ reciprocal rescalings on $X$.
\qed\end{remark}

\textbf{Rescaling and compactness. }Consider how compactness of $X$ is reflected on the properties of $X_{L}$.

\begin{proposition}
Let $L$ be a bi-function on a compact Hausdorff space $X$ which satisfies one of the conditions of Proposition \ref{ocom}. Then $X_{L}$ has finite number of components and every component has finite radius.
\end{proposition}
\begin{proof}
Since from Proposition \ref{ocom} the components of $X_{L}$ are open disjoint sets which cover $X$, it follows that there are only finitely many of them. Let us prove the second claim. Without loss of generality we may assume that $X_{L}$ is connected.

For $x\in X$ and $n\in\N$ let $V_{x}^{n}$ be the closed $n$-th neighborhood of $x$ in $X_{L}$. Clearly, $x\in V_{x}^{n}\subset V_{x}^{n+1}$, and if $y\in V_{x}^{n}$, then $V_{y}^{m}\subset V_{x}^{m+n}$, for any $m\in\N$.

Recall that if either of the conditions of Proposition \ref{ocom} hold, then $y\in \Int V_{y}^{2}$, for any $y$, and so $V_{x}^{n}\subset \Int V_{x}^{n+2}$. Therefore, $\bigcup\limits_{n\in\N}\Int V_{x}^{n}=\bigcup\limits_{n\in\N} V_{x}^{n}=X$, where the last equality follows from the fact that $X_{L}$ is connected. Hence, $\left\{\Int V_{x}^{n}\right\}_{n\in\N}$ is an increasing sequence of open sets, that cover $X$. Since $X$ is compact we conclude that there is $n\in\N$ such that $X=V_{x}^{n}$, and so the radius of $X_{L}$ is at most $n$.
\end{proof}

Combining this proposition with Proposition \ref{radius} we obtain the following fact.

\begin{corollary}\label{comp}
Let $L$ be a bi-function on a compact Hausdorff space $X$ which satisfies one of the conditions of Proposition \ref{ocom}. Then there is $l\in\N$ such that any symmetric bi-function $M$ on $X$ is a $\pm1$ symmetric rescaling of $L$ whenever $\det_{L}=\det_{M}$ on all of the subsets of $X$ of cardinality up to $l$.
\end{corollary}

\begin{proposition}
Let $L$ be a continuous non-degenerate bi-function on a compact Hausdorff space $X$. Then there is $l\in\N$ such that $X_{L}$ does not contain induced cycles of length exceeding $l$.
\end{proposition}
\begin{proof}
Since $L$ is continuous and non-degenerate, for every $x$ there is an open neighborhood $V^{x}$ of $x$ such that $L\left(y,z\right)\ne 0$, for every $y,z\in V^{x}$. Then $V^{x}_{L}$ is a complete subgraph of $X_{L}$. Since $X$ is compact there are $x_{1},...,x_{n}\in X$ such that $X=\bigcup\limits_{i=1}^{n}V^{x_{i}}$. Assume that $y_{0},y_{1},...,y_{m}=y_{0}$ is an induced cycle in $X_{L}$. Let $i\in \overline{0,m-1}$ and let $y_{i}\in V_{x_{k}}$, for some $k\in \overline{1,n}$. Then $y_{i}$ is not joined by an edge with $y_{j}$, unless $j=i\pm 1 \mod n$, and so the cardinality of the set $\left\{j\in \overline{0,m-1}\left|y_{j}\in V^{x_{k}}\right.\right\}$ is at most $3$. Hence, $m\le 3n$, and so the lengths of the induced cycles in $X^{L}$ does not exceed $l=3n$.
\end{proof}

The following example shows that continuity cannot be replaced by separate continuity (even together with continuity on the diagonal) in the statement of the preceding proposition.

\begin{example}
Let $X=\left[0,1\right]$. Let $Y=X\times X\backslash \left\{\left(0,0\right),\left(1,1\right)\right\}$. Consider \linebreak $Z=\left\{\left(x,y\right)\in Y\left|y\le x^{2},~ x^{2}+\left(1-y\right)^{2}\le 1\right.\right\}\cup\left\{\left(x,y\right)\in Y\left|x\le y^{2},~ y^{2}+\left(1-x\right)^{2}\le 1\right.\right\}$, which is a closed subset of $Y$. It is not difficult to construct a continuous function $L:Y\to \left[0,1\right]$ such that $L^{-1}\left(0\right)=Z$, $L\left(x,y\right)=L\left(y,x\right)$, for any $\left(x,y\right)\in Y$ and $$L\left(x,0\right)=L\left(x,1\right)=L\left(0,x\right)=L\left(1,x\right)=L\left(x,x\right)=L\left(0,1\right)=L\left(1,0\right)=1,$$ for every $x\in \left(0,1\right)$. Extending $L$ to be defined on $X\times X$ by $L\left(0,0\right)=L\left(1,1\right)=1$ we obtain a separately continuous and continuous on the diagonal non-degenerate bi-function on $\left[0,1\right]$. We will show however that $X_{L}$ contains induced cycles of arbitrary length.

Fix $x\in \left(0,1\right)$ and $y\in \left(x^{2},x\right)$. Consider a sequence $x, y, x^{2}, y^{2}, x^{4}, y^{4},..., z$ where $z$ is the first element of the sequence such that $x^{2}+\left(1-z\right)^{2}> 1$. It is easy to see that we have obtained an induced cycle in $X_{L}$, of length approximately equal to $2t$, where $x^{2^{t}}=1-\sqrt{1-x^{2}}$. Then $2^{t}=\frac{\ln \left(1-\sqrt{1-x^{2}}\right)}{\ln x}$, and since $$\lim\limits_{x\to 1-}\frac{\ln \left(1-\sqrt{1-x^{2}}\right)}{\ln x}=\lim\limits_{x\to 1-}\frac{\frac{\frac{x}{\sqrt{1-x^{2}}}}{1-\sqrt{1-x^{2}}}}{\frac{1}{x}}=\lim\limits_{x\to 1-}\frac{1}{\sqrt{1-x^{2}}}=+\8,$$
we conclude that the length of the considered cycle grows infinitely, as we choose $x$ closer to $1$.
\qed\end{example}

\textbf{Rescaling and connectedness. }Let us bring connectedness into the picture.

\begin{proposition}\label{simeq}
Let $L$ be a bi-function on a connected topological space $X$ which satisfies one of the conditions of Proposition \ref{ocom}. Then:
\item[(i)] $X_{L}$ is connected.
\item[(ii)] If $M$ is a rescaling of $L$, then for every $x\in X$ there are unique $f,g:X\to\Cp$ such that $\left(f,g\right)$ rescales $L$ to $M$ with $f\left(x\right)=1$.
\end{proposition}
\begin{proof}
(i): From Proposition \ref{ocom} the components of $X_{L}$ are open disjoint sets, which cover $X$. Since $X$ is connected it follows that there is just one of them.

(ii) follows from (i) and part (ii) of Corollary \ref{condis}.
\end{proof}
\begin{remark}
Part (ii) can be extended by adding other claims from Corollary \ref{condis}.\qed
\end{remark}

\begin{proposition}\label{conr}Let $L$ and $M$ be bi-functions on a connected space $X$. Then:
\item[(i)] If $L$ and $M$ are non-degenerate separately continuous rescalings, and there is a discrete subgroup $\Gamma$ of $\Cp$ such that for every $x,y\in X$ there is $\gamma\in\Gamma$ so that $M\left(x,y\right)=\gamma L\left(x,y\right)$. Then $M=\gamma L$, for some $\gamma\in\Gamma$.
\item[(ii)] If $L$ and $M$ satisfy the same of the conditions of Proposition \ref{ocom}, and are symmetric $\Gamma$-rescalings for a discrete subgroup $\Gamma$ of $\Cp$, then $M=\gamma L$, for some $\gamma\in\Gamma$. In particular, if $L$ and $M$ are symmetric with $\det_{L}=\det_{M}$, then $L=M$.
\end{proposition}
\begin{proof}
(i): From part (i) of Proposition \ref{cresc} there are $f,g:X\to \Gamma$ such that $\left(f,g\right)$ rescales $L$ to $M$. Then $f$ and $g$ are continuous, according to Proposition \ref{con}. Since a continuous function on a connected space with values in a discrete set has to be constant, the result follows.

(ii): The first claim is proven analogously to (i). The last claim follows from the first claim and Theorem \ref{similar}.
\end{proof}

\begin{remark}\label{conrr}In particular, it follows that if $L$ and $M$ are non-degenerate rescalings with $L^{2}=M^{2}$, then either $L=M$, or $L=-M$.\qed\end{remark}

Combining part (ii) with Corollary \ref{comp} we get a stronger version of the latter.

\begin{corollary}
Let $L$ be a symmetric bi-function on a connected compact Hausdorff space $X$ which satisfies one of the conditions of Proposition \ref{ocom}. Then there is $l\in\N$ such that any symmetric bi-function $M$ on $X$ which satisfies the same of the conditions of Proposition \ref{ocom} is equal to $L$ whenever $\det_{L}=\det_{M}$ on all of the subsets of $X$ of cardinality up to $l$.
\end{corollary}

\textbf{A criterion for rescaling. }In this subsection we will give a criterion for two separately continuous functions to be rescalings based on equality \ref{*}. As was already mentioned, if $L$ and $M$ are rescalings, then \ref{*} holds for any $x,y,z$. Also note that if $L$ and $M$ are non-degenerate, and the equality \ref{*} holds for any $x,y,z$ then plugging in $x=z$ we see that $U_{y}\cap U^{y}$ is the same, for every $y\in Y$, with respect to $L$ and $M$. Moreover, in this case $\left(f_{y},g_{y}\right)$ rescales $L$ to $M$ on $U_{y}\cap U^{y}$, where $f_{y}=\frac{M\left(\cdot,y\right)L\left(y,y\right)}{L\left(\cdot,y\right)M\left(y,y\right)}$ and $g_{y}=\frac{M\left(y,\cdot\right)}{L\left(y,\cdot\right)}$. Hence, in this case $U_{y}\cap U^{y}\in \mathcal{Y}$, in the notations of Lemma \ref{zorn}. This allows us to obtain the following result.

\begin{theorem}\label{qr}
Let $X$ be a topological space and let $L$ and $M$ be non-degenerate bi-functions on $X$ such that \ref{*} is satisfied for every $x,y,z\in X$. Then $L$ and $M$ are rescalings whenever one of the following conditions is fulfilled:
\item[(i)] $L$ and $M$ are separately continuous, and there is $z\in X$ such that $\overline{U_{z}}=\overline{U^{z}}=X$ ($U_{z}$ and $U^{z}$ with respect to $L$);
\item[(ii)] $X$ is connected and locally connected, $L$ is separately continuous and $\overline{U_{z}}=\overline{U^{z}}=X$ for every $z\in X$.
\end{theorem}
\begin{proof}
First, note that sufficiency of (i) follows from Proposition \ref{conde}, since if \ref{*} is satisfied for every $x,y,z\in X$, then $L$ and $M$ are rescalings on $U_{y}\cap U^{y}$, for any $y\in X$. Hence, $L$ and $M$ are rescalings in a neighborhood of any point, and also, they are rescalings on a dense set $U_{z}\cap U^{z}$. Let us now prove sufficiency of (ii).

Let $\mathcal{U}$ stand for the class of all open connected subsets $U$ of $X$, such that $L$ and $M$ are rescalings on $U$. Combining the fact that the union of an increasing collection of open connected sets is open and connected with Lemma \ref{zorn}, we see that $\mathcal{U}$ satisfies the conditions of Zorn's lemma.

For every $x\in X$ define $V_{x}=U_{x}\cap U^{x}$, which is an open neighborhood of $x$. Note also that from the discussion above $L$ and $M$ are rescalings on $V_{x}$, and so any component of $V_{x}$ belongs to $\mathcal{U}$ (since $X$ is locally connected any component of an open set is open). Since both $U_{x}$ and $U^{x}$ are dense, then $\left(X\backslash U_{x}\right)\cup\left(X\backslash U^{x}\right)$ is nowhere dense, and so $V_{z}$ is dense (and open).\medskip

Using Zorn's Lemma we can choose a maximal element $U$ of $\mathcal{U}$. Since $X$ is connected, in order to prove that $U=X$ it is enough to show that $\partial U=\varnothing$. Assume that $w\in \partial U$. Then $V_{w}$ is an open neighborhood of $w\in\partial U\subset \overline{U}$, and so $U\cap V_{w}$ is a nonempty open set. Choose $v$ in the latter set; it follows that $v,w\in V_{v}$. Let $W$ be the component of $V_{v}$ that contains $w$. Then $W$ is a connected open neighborhood of $w$, and so it intersects with $U$. Since $U$ is connected, $U\cup W$ is also connected. If we show that $U\cup W\in\mathcal{U}$, we will reach a contradiction with the maximality of $U$, since $U\cup W$ contains $U$ and also contains $w\in \partial U\subset X\backslash U$.\medskip

Since $U\in\mathcal{U}$, there are $f,g:U\to\C^{\times}$ such that $\left(f,g\right)$ rescales $L$ to $M$ on $U$ with $f\left(v\right)=1$; we also have that $\left(f_{v},g_{v}\right)$ rescales $L$ to $M$ on $V_{v}\supset W$ and $f_{v}\left(v\right)=1$. Note that $v\in U\cap W\subset V_{v}$, and since both pairs $\left(f,g\right)$ and $\left(f_{v},g_{v}\right)$ rescale $L$ to $M$ on $U\cap W$ with $f\left(v\right)=f_{v}\left(v\right)$ it follows that $f,g$ and $f_{v},g_{v}$ agree on $U\cap W$. Hence, there are common extensions of $f$ and $f_{v}$ and of $g$ and $g_{v}$ on $U\cup W$ (we will also denote it by $f$ and $g$).\medskip

Now, for $x,y\in U\cup W$ choose $z\in V_{x}\cap V_{y}\cap W\cap U$ (this set is nonempty since $V_{x}$ and $V_{y}$ are dense open sets and $U\cap W$ is open and nonempty). Then, each of $L\left(x,z\right), L\left(y,z\right), M\left(z,z\right)$ are not zero, and from \ref{*} we have that
\begin{align*}
M\left(x,y\right)&=\frac{M\left(x,z\right)}{L\left(x,z\right)} \frac{M\left(z,y\right)}{L\left(z,y\right)} \frac{L\left(z,z\right)}{M\left(z,z\right)}L\left(x,y\right)\\ &=\frac{f\left(x\right)g\left(z\right)f\left(z\right)g\left(y\right)}{f\left(z\right)g\left(z\right)}L\left(x,y\right)=f\left(x\right)g\left(y\right)L\left(x,y\right).
\end{align*}
Since $x,y$ were chosen arbitrarily, we conclude that $\left(f,g\right)$ rescales $L$ to $M$ on $U\cup W$, and so we have reached a contradiction.
\end{proof}

For $x,y,z\in X$ consider equalities
\begin{equation}\label{*'}
M\left(x,y\right)M\left(z,y\right)L\left(x,z\right)L\left(y,y\right)=L\left(x,y\right)L\left(z,y\right)M\left(x,z\right)M\left(y,y\right),\tag*{($\ast '$)}
\end{equation}
\begin{equation}\label{*''}
M\left(x,y\right)\overline{M\left(z,y\right)}L\left(x,z\right)L\left(y,y\right)=L\left(x,y\right)\overline{L\left(z,y\right)}M\left(x,z\right)M\left(y,y\right).\tag*{($\ast ''$)}
\end{equation}

It is clear that if $L$ and $M$ are symmetric (or Hermitean) rescalings, then \ref{*'} (or \ref{*''}) is satisfied for any $x,y,z\in X$. Also, note that applying the equations above to $L'$ and $M'$ or $\overline{L}$ and $\overline{M}$ we obtain other variations of these equations. Adopting the proof of the preceding theorem and using a suitable variation of Lemma \ref{zorn} one can show the following criteria.

\begin{proposition}
Let $X$ be a topological space and let $L$ and $M$ be non-degenerate bi-functions on $X$ such that \textup{\ref{*'}} is satisfied for any
$x,y,z\in X$ (or \textup{\ref{*''}} is satisfied for any
$x,y,z\in X$ and there is $w\in X$ with $\frac{\widehat{M}\left(w\right)}{\widehat{L}\left(w\right)}>0$). Then $L$ and $M$ are symmetric (or Hermitean) rescalings once one of the following conditions is fulfilled:
\item[(i)] $L$ and $M$ are separately continuous, and there is $z\in X$ such that $\overline{U_{z}}=X$;
\item[(ii)] $X$ is connected and locally connected, $L$ is separately continuous and $\overline{U_{z}}=X$ for every $z\in X$.
\end{proposition}

\begin{remark}
The reason for an additional condition for Hermitean rescalings is that we have to guarantee that $\frac{\widehat{M}}{\widehat{L}}>0$ so that $f_{y}$ Hermiteanly rescales $L$ to $M$ on $U_{y}$ with $f_{y}\left(y\right)>0$, where $f_{y}=\frac{M\left(\cdot,y\right)}{L\left(\cdot,y\right)}\sqrt{\frac{L\left(y,y\right)}{M\left(y,y\right)}}$. The existence of $w\in X$ with $\frac{\widehat{M}\left(w\right)}{\widehat{L}\left(w\right)}>0$ implies $\frac{\widehat{M}}{\widehat{L}}>0$. Indeed, substituting $x=y=z$ we get that $\frac{\widehat{M}}{\widehat{L}}$ is real-valued, while substituting $x=z$ we get that $\frac{\widehat{M}\left(x\right)}{\widehat{L}\left(x\right)}\frac{\widehat{M}\left(y\right)}{\widehat{L}\left(y\right)}>0$, whenever $x$ and $y$ are joined with an edge in $X_{L}$. Using $\frac{\widehat{M}\left(w\right)}{\widehat{L}\left(w\right)}>0$ and the equivalence relation argument as in Section \ref{bf}, we conclude that $\frac{\widehat{M}}{\widehat{L}}>0$.
\qed
\end{remark}

Combining this with Corollary \ref{nv} and part (i) of Proposition \ref{conr} we get that if $L$ and $M$ are non-degenerate symmetric bi-functions on a connected topological space $X$ that satisfy either of the conditions of Corollary \ref{nv}, and also either of the conditions of the preceding proposition, then $L=M$. Consider an example that shows that the existence of $z$ with $\overline{U_{z}}=X$ is essential. In fact, we will construct two real-valued symmetric bi-functions on a closed interval, which satisfy the equality from part (i) of Corollary \ref{nv}, for every $x,y,z$, but are not rescalings.

\begin{example}\label{exa}
Let $K_{0}$ be an arbitrary continuous non-vanishing bi-function on $\left(-1,1\right)$, such that $\lim\limits_{\left|x\right|\to 1}\widehat{K_{0}}\left(x\right)= 0$, e.g. $K_{0}\left(x,y\right)=\left(1-x^{2}\right)\left(1-y^{2}\right)$, and extend it on $\R\times \R$ by $0$. For $n\in\Z$ define $K_{n}:\R\times \R\to \R$ by $K_{n}\left(x,y\right)=K_{0}\left(x-n,y-n\right)$, and let $K=\sum\limits_{n\in\Z}K_{n}$. It is clear that locally the sum is of at most two non-zero summands, and so $K$ is continuous.

Observe that if $x\le n$ and $y\ge n+1$, for some $n$, then $K\left(x,y\right)=0$. Therefore, the set of points in $\R\times \R$ where $K$ is not zero is a ``ladder'' that may be viewed as a continuous analogue of $L$ from Example \ref{ex}. Adding extra regions where $K$ does not vanish in a way similar to the construction of $L_{A}$ from that example will increase the complexity of the graph of $K$. Here we will only construct an analogue of $L_{4}^{\pm}$, but it should be clear how to construct analogues of $L_{n}^{\pm}$ or general $L_{A}$.

Let $f_{4}^{\pm}:\left[0,4\right]\to\R$ be a continuous function defined by $$f_{4}^{\pm}\left(x\right)=\left\{\begin{array}{lll} \pm\left(1-2x\right) & 0\le x\le\frac{1}{2} \\ 0 & \frac{1}{2}\le x\le \frac{7}{2}  \\ 2x-7 & \frac{7}{2}\le x\le 4  \end{array}\right.$$
Let $M^{\pm}=K\left|_{\left[0,4\right]\times \left[0,4\right]}\right.+f_{4}^{\pm}\otimes f_{4}^{\pm}$. Observe that $$M^{\pm}\left(x,y\right)=\left\{\begin{array}{lll} K\left(x,y\right) & \frac{1}{2}\le x\le \frac{7}{2}\mbox{ or }\frac{1}{2}\le y\le \frac{7}{2} \\ K\left(x,y\right)+f^{+}\left(x\right)f^{+}\left(y\right) & x,y\le \frac{1}{2}\mbox{ or }x,y\ge\frac{7}{2}  \\
\pm f^{+}\left(x\right)f^{+}\left(y\right) & x\le \frac{1}{2}\mbox{ and }y\ge\frac{7}{2}\mbox{, or }y\le \frac{1}{2}\mbox{ and }x\ge\frac{7}{2}  \end{array}\right.$$
Thus, $M^{+}$ and $M^{-}$ are non-equal continuous bi-functions with $\left|M^{+}\right|=\left|M^{-}\right|$. Hence, they are not rescalings, according to Remark \ref{conrr}. However, we will show that the equality from Corollary \ref{nv} holds for $M^{+}$, $M^{-}$ and any $x,y,z\in X$. Without loss of generality we may assume that $x\le y\le z$.\medskip

If $x\ge \frac{1}{2}$, or $z\le\frac{7}{2}$ we have $M^{+}\left(x,y\right)=M^{-}\left(x,y\right)$, $M^{+}\left(y,z\right)=M^{-}\left(y,z\right)$, $M^{+}\left(z,x\right)=M^{-}\left(z,x\right)$, and so the equality holds. If $y\le \frac{1}{2}$ and $z\ge\frac{7}{2}$ we have $M^{+}\left(x,y\right)=M^{-}\left(x,y\right)$, $M^{+}\left(y,z\right)=-M^{-}\left(y,z\right)$ and $M^{+}\left(z,x\right)=-M^{-}\left(z,x\right)$, and so the equality holds. The case when $y\ge \frac{7}{2}$ and $x\le\frac{1}{2}$ is analogous. Finally, if $x<\frac{1}{2}$, $\frac{1}{2}\le y\le \frac{7}{2}$ and $z>\frac{7}{2}$, then $M^{\pm}\left(x,y\right)=K\left(x,y\right)$, $M^{\pm}\left(y,z\right)=K\left(y,z\right)$, and either $K\left(x,y\right)=0$ or $K\left(y,z\right)=0$ depending on whether $y\ge 2$ or $y\le 2$. Hence, the equality holds.
\qed\end{example}

\section{Rescaling of sesqui-holomorphic functions}\label{resq}

Let $X$ be a domain in $\C^{n}$, i.e. an open connected set. In this section we perform an analogous investigation to the preceding section, but this time instead of continuity we will be considering holomorphicity. Let us start with the following fact (the proof is similar to the proof of Proposition \ref{con}).

\begin{proposition}
Let $X$ be a domain in $\C^{n}$ and let $L$ and $M$ be bi-functions on $X$, which are (anti-) holomorphic in the first variable, and for every $x\in X$ there is $y\in X$ such that $L\left(x,y\right)\ne 0$. If $f,g:X\to\Cp$ are such that $\left(f,g\right)$ rescales $L$ to $M$, then $f$ is (anti-) holomorphic.
\end{proposition}

It is clear that the same result holds for the second variable and $g$.\medskip

Observe that by Hartog's theorem (see \cite[VII.4, Theorem 4]{bm}) a bi-function $L$ on $X$ is holomorphic if and only if it is holomorphic in each variable. In this case the diagonal function $\widehat{L}$ is also holomorphic. Consider another class of bi-functions. A bi-function $L$ on $X$ is called \emph{sesqui-holomorphic} if it is holomorphic in the first variable and anti-holomorphic in the second. Let $X^{*}=\left\{\overline{x}\left|x\in X\right.\right\}\subset \C^{n}$. By Hartog's theorem $L$ is sesqui-holomorphic if and only if the function $M:X\times X^{*}\to \C$ defined by $M\left(x,y\right)=L\left(x,\overline{y}\right)$ is holomorphic. Hence, any sesqui-holomorphic bi-function is continuous.

In this section we will heavily rely on the Uniqueness Principle. Namely, recall that if two (anti-) holomorphic functions coincide on a somewhere dense (i.e. whose closure has a nonempty interior) subset of their domain, then they are equal (see \cite[II.2, Theorem 4]{bm}). Hence, if $L$ is a (sesqui-) holomorphic non-degenerate bi-function, then both $U_{y}=\left\{x\in X\left|L\left(x,y\right)\ne 0\right.\right\}$ and $U^{y}=\left\{x\in X\left|L\left(y,x\right)\ne 0\right.\right\}$ are dense in $X$, for any $y\in X$. Indeed, if $\overline{U_{y}}\ne X$ (or $\overline{U^{y}}\ne X$), then the (anti-) holomorphic function $L\left(\cdot,y\right)$ (or $L\left(y,\cdot\right)$) vanishes on an open set, and so it is identically zero, which contradicts $L\left(y,y\right)\ne 0$. Another consequence is that being rescalings is determined by the behavior on a somewhere dense set.

\begin{proposition}
Let $L$ and $M$ be (sesqui-) holomorphic non-degenerate bi-functions on a domain $X\subset \C^{n}$. If $L$ and $M$ are rescalings on a somewhere dense subset of $X$, then $L$ and $M$ are rescalings. In particular, if there is a somewhere dense connected $U\subset X$ and $f,g:U\to\Cp$ such that $\left(f,g\right)$ rescales $L$ to $M$ on $U$, then $f$ and $g$ can be extended on $X$ so that the extension $\left(f,g\right)$ rescales $L$ to $M$ on $X$.
\end{proposition}
\begin{proof}
Let $Y=X\times X$. Define a (sesqui-) holomorphic function $N:Y\times Y\to\C$ by
$$N\left(x,y,z,w\right)=M\left(x,w\right)M\left(y,z\right)L\left(x,z\right)L\left(y,w\right)-L\left(x,w\right)L\left(y,z\right)M\left(x,z\right)M\left(y,w\right).$$
As $L$ and $M$ are rescalings on a somewhere dense subset $V$ of $X$, a simple calculation shows that $N$ vanishes on $V\times V$, which is somewhere dense in $Y$. Therefore, from the Uniqueness Principle $N\equiv 0$, and substituting $y=w$ in this equality we get that \ref{*} holds for any $x,y,z\in X$. Since $L$ and $M$ satisfy the conditions of Theorem \ref{qr}, we conclude that $L$ and $M$ are rescalings.

Now assume that there is a somewhere dense connected $U\subset X$ and $f,g:U\to\Cp$ such that $\left(f,g\right)$ rescales $L$ to $M$ on $U$. Then $L$ and $M$ are rescalings, and so there are $f_{1},g_{1}:X\to\Cp$ such that $M=f_{1}\otimes g_{1} L$ and $f_{1}\left(x\right)=f\left(x\right)$, for some $x\in U$. Since both $\left(f,g\right)$ and $\left(f_{1},g_{1}\right)$ rescale $L$ to $M$ on a connected set $U$, and $L$ and $M$ are continuous, it follows from part (ii) of Proposition \ref{simeq} that $f=f_{1}$ and $g=g_{1}$ on $U$. Hence, $f_{1}$ and $g_{1}$ are extensions of $f$ and $g$ respectively.
\end{proof}

\begin{remark}\label{logl}
If in the preceding proposition $f=g$, then $f_{1}=g_{1}$, due to the Uniqueness Principle. Hence, if $L$ and $M$ are holomorphic non-degenerate bi-functions on $X$, which are symmetric rescalings on a somewhere dense $U\subset X$, then they are symmetric rescalings. The analogous is also true for reciprocal rescalings, and for Hermitean rescalings of sesqui-holomorphic bi-functions.
\qed\end{remark}

We can now state a criterion for being symmetric rescalings in the class of holomorphic bi-functions.

\begin{proposition}\label{rigs}
Let $L$ and $M$ be holomorphic non-degenerate bi-functions on a domain $X$. Then $L$ and $M$ are symmetric rescalings if and only if there is a somewhere dense set $U\subset X$ such that $\frac{L\left(x,y\right)^{2}}{L\left(x,x\right)L\left(y,y\right)}=\frac{M\left(x,y\right)^{2}}{M\left(x,x\right)M\left(y,y\right)}$, for every $x,y\in U$.
\end{proposition}
\begin{proof}
Necessity follows from substituting $x=z$ into \ref{*'}. Let us show sufficiency. Since both $L$ and $M$ are continuous and non-degenerate, by replacing $U$ with $\Int \overline{U}$ we may assume that $U$ is a nonempty open set.

Since $M$ and $L$ are non-degenerate, $\frac{\widehat{M}}{\widehat{L}}$ is a holomorphic non-vanishing function on $X$. Fix $z\in U$. Let $V\subset U$ be a neighborhood of $z$ such that $M\left(x,y\right)\ne 0\ne L\left(x,y\right)$, for any $x,y\in V$.

There is an open neighborhood $W\subset V$ of $z$ and holomorphic $f:W\to\Cp$ such that $\frac{\widehat{M}}{\widehat{L}}=f^{2}$ on $W$. Then, for any $x,y\in W$ we have $\frac{f\left(x\right)^{2}f\left(y\right)^{2}L\left(x,y\right)^{2}}{M\left(x,y\right)^{2}}=1$, and so $\frac{f\otimes fL}{M}$ is a holomorphic bi-function on $W$ with values in $\left\{-1,1\right\}$. Hence, either $M=f\otimes fL$ on $W$, or $M=-f\otimes fL$. In the first case $f$ symmetrically rescales $L$ to $M$ on $W$, while in the second $\ii f$ symmetrically rescales $L$ to $M$ on $W$. Thus, from the preceding remark, $L$ and $M$ are symmetric rescalings.
\end{proof}

In order to obtain a similar characterization for Hermitean rescalings of sesqui-holomorphic bi-functions we need a special Uniqueness Principle (it follows from \cite[II.4, Theorem 7]{bm}; see also \cite{nik}).

\begin{proposition}\label{rig} Let $L$ and $M$ be sesqui-holomorphic bi-functions on a domain $X\subset \C^{n}$. If $\widehat{L}=\widehat{M}$ on a somewhere dense set $U\subset X$, then $L=M$.
\end{proposition}

From applying the proposition to $L$ and $L^{*}$, we get that a sesqui-holomorphic bi-function $L$ is Hermitean if and only if $\widehat{L}$ is real-valued on a somewhere dense set. This observation leads to the following result.

\begin{lemma}\label{real}
Let $L$ and $M$ be non-degenerate sesqui-holomorphic bi-functions on a domain $X$. If $\frac{\widehat{M}}{\widehat{L}}$ is real-valued on a somewhere dense set, then either $\frac{\widehat{M}}{\widehat{L}}>0$ on $X$, or $\frac{\widehat{M}}{\widehat{L}}<0$ on $X$.
\end{lemma}
\begin{proof}
We may assume that there is a nonempty open set $U\subset X$ such that $\frac{\widehat{M}}{\widehat{L}}$ is real-valued on $U$. Fix $z\in U$. Since $L\left(z,z\right)\ne 0\ne M\left(z,z\right)$ and $L$ and $M$ are continuous, there is an open connected neighborhood $V\subset U$ of $z$ such that neither $L$ nor $M$ vanish on $V\times V$. Then, $K=\frac{M}{L}$ is a sesqui-holomorphic bi-function on $V$, such that $\widehat{K}$ is real-valued. From the discussion above it follows that $K$ is Hermitean, and so $ML^{*}=M^{*}L$ on $V$. Since both of these bi-functions are sesqui-holomorphic, from the (regular) Uniqueness Principle $ML^{*}=M^{*}L$ on $X$, and so $M\left(x,x\right)\overline{L\left(x,x\right)}=\overline{M\left(x,x\right)}L\left(x,x\right)$, for every $x\in X$. Therefore, $\frac{\widehat{M}}{\widehat{L}}$ is real-valued on $X$. Since $X$ is connected, and $\frac{\widehat{M}}{\widehat{L}}$ is continuous and non-vanishing, it is either always positive, or always negative.
\end{proof}

Now we can state a criterion of rescaling of sesqui-holomorphic functions.

\begin{theorem}\label{rrig}
Let $L$ and $M$ be sesqui-holomorphic non-degenerate bi-functions on a domain $X\subset\C^{n}$. The following are equivalent:
\item[(i)] $L$ and $M$ are Hermitean rescalings (on a somewhere dense subset of $X$);
\item[(ii)] There is a nonempty open set $U\subset X$ and a holomorphic function $f:U\to\Cp$, such that $\widehat{M}=\left|f\right|^{2}\widehat{L}$ on $U$;
\item[(iii)] There is a nonempty open set $U\subset X$ such that $\frac{\partial^{2}}{\partial z_{j}\partial \overline{z_{k}}}\log \widehat{L}=\frac{\partial^{2}}{\partial z_{j}\partial \overline{z_{k}}}\log \widehat{M}$, on $U$, for all $j,k\in\overline{1,n}$, and $\frac{\widehat{M}}{\widehat{L}}>0$ on a somewhere dense set;
\item[(iv)] There are $y,z\in X$ and a somewhere dense set $U\subset X$ such that $\frac{M\left(z,z\right)}{L\left(z,z\right)}>0$ and $\frac{\left|L\left(x,y\right)\right|^{2}}{L\left(x,x\right)L\left(y,y\right)}=\frac{\left|M\left(x,y\right)\right|^{2}}{M\left(x,x\right)M\left(y,y\right)}$, for every $x\in U$.
\end{theorem}

Note that the local equivalence of (i) and (iii) is known as Calabi Rigidity (see \cite{cal}). Also note that in (iv) we do not require $y,z\in U$.

\begin{proof}
From Remark \ref{logl} being Hermitean rescalings is the same as being Hermitean rescalings on somewhere dense subset of $X$. (i)$\Rightarrow$(ii) is trivial; both (i)$\Rightarrow$(iii) and (i)$\Rightarrow$(iv) follow from a simple calculation. In fact, (iv) is just a rearrangement of \ref{*''} for $x=z$. (ii)$\Rightarrow$(i) follows from Proposition \ref{rig} applied to $M$ and $f\otimes \overline{f} L$.

(iii)$\Rightarrow$(ii): From the preceding lemma, the conditions of (ii) mean that $\log \frac{\widehat{M}}{\widehat{L}}$ is pluriharmonic on some open set $V$. Then, there is a nonempty open set $W\subset V$ and a holomorphic function $f:V\to\Cp$, such that $\left|f\right|^{2}=\frac{\widehat{M}}{\widehat{L}}$ on $W$ (see \cite[Theorem 4.4.4]{rudin}).\medskip

(iv)$\Rightarrow$(ii): We may assume that $U$ is open. Since $U_{y}$ is dense in $X$, it follows that $V=U\cap U_{y}$ is a nonempty open set. Let $\alpha= \frac{M\left(y,y\right)}{L\left(y,y\right)}\ne 0$. For any $x\in V$ we have $\frac{\alpha M\left(x,x\right)}{L\left(x,x\right)}=\left|\frac{M\left(x,y\right)}{L\left(x,y\right)}\right|^{2}$, and so $\frac{\alpha\widehat{M}}{\widehat{L}}$ is positive on $V$. Therefore, from the preceding lemma $\frac{\alpha\widehat{M}}{\widehat{L}}$ is positive on $X$. In particular, we have that $\alpha^{2}=\alpha\frac{M\left(y,y\right)}{L\left(x,y\right)}>0$ and so $\alpha\in\R$. Also, $\alpha\frac{M\left(z,z\right)}{L\left(z,z\right)}>0$, and since $\frac{M\left(z,z\right)}{L\left(z,z\right)}>0$ it follows that $\alpha>0$.

Define $f:V\to\Cp$ by $f\left(x\right)=\alpha^{-\frac{1}{2}}\frac{M\left(x,y\right)}{L\left(x,y\right)}$. Then for $x\in V$ we have \linebreak $\left|f\left(x\right)\right|^{2}=\alpha^{-1}\left|\frac{M\left(x,y\right)}{L\left(x,y\right)}\right|^{2}=\alpha^{-1} \frac{\alpha M\left(x,x\right)}{L\left(x,x\right)}$, and so $\widehat{M}=\left|f\right|^{2}\widehat{L}$ on $V$.
\end{proof}

\begin{remark}
It is easy to see that all the results in this section are valid when $X$ is a general complex manifold.
\qed\end{remark}

\section{Geometric interpretation}\label{geoin}

In this section we will consider geometric consequences of Theorem \ref{similar}. While general bi-functions are devoid of geometric meaning, we will focus on a more special class of them.\medskip

\textbf{Positive semi-definite kernels. }A bi-function $K$ on $X$ is called \emph{positive (semi-) definite} if for any $x_{1},...,x_{n}\in X$, the matrix $\left[K\left(x_{i},x_{j}\right)\right]_{i,j=1}^{n}$ is positive (semi-) definite. Note that positive (semi-) definite bi-functions are traditionally called positive (semi-) definite \emph{kernels}. From the information about positive (semi-) definite matrices, we deduce the following list of properties:
\begin{itemize}
\item From Sylvester's criterion (see \cite[Theorem 7.2.5]{hj}) $K:X\times X\to \C$ is positive (semi-) definite if and only if it is Hermitean and $\det_{K}$ is a positive (non-negative) function on $Fin\left(X\right)$; in this case $K'=\overline{K}$ is also positive (semi-) definite, as well as $\re K$ and $\alpha K$, for any $\alpha >0$.
\item A sum of a positive semi-definite kernel and a positive (semi-) definite kernel is positive (semi-) definite; by Schur's Theorem (see \cite[Theorem 7.5.3]{hj}) a product of positive (semi-) definite kernels is positive (semi-) definite.
\item A pointwise limit of positive semi-definite kernels is positive semi-definite.
\end{itemize}

It also follows from Proposition \ref{pdc} that any Hermitean rescaling of a positive (semi-) definite kernel is positive (semi-) definite; in particular, $\diag f$ is positive (semi-) definite for every $f:X\to\left(0,+\8\right)$ ($f:X\to\left[0,+\8\right)$), and $f\otimes\overline{f}$ is positive semi-definite for any $f:X\to \C$.

Analogously to how a positive semi-definite matrix is a Gram matrix of a collection of vectors in a Euclidean space, a similar representation holds for positive semi-definite kernels. By a Hilbert space we will mean a finite- or infinite-dimensional complete inner product space over either $\R$ or $\C$. We will use the term \emph{unitary} operator for a surjective isometric operator on a (real or complex) Hilbert space. The following is a version of a classic result from \cite{aron}.

\begin{theorem}[Moore-Aronszajn]
Let $X$ be a set. A complex-valued (real-valued) bi-function $K$ on $X$ is positive semi-definite if and only if there is a complex (real) Hilbert space $H$ and a map $\kappa:X\to H$, such that $\overline{\spa~\kappa\left(X\right)}=H$ and $K\left(x,y\right)=\left<\kappa\left(x\right),\kappa\left(y\right)\right>$, for every $x,y\in X$. Furthermore, the pair of $H$ and $\kappa$ is unique up to the unitary equivalence, i.e. if $H_{i}$ and $\kappa_{i}$ satisfy the conditions above, for $i=1,2$, then there is a unitary operator $T:H_{1}\to H_{2}$ such that $T\kappa_{1}=\kappa_{2}$.
\end{theorem}

\begin{remark}\label{ma} In the notations of the theorem, $K$ is non-degenerate if and only if $0_{H}\not\in\kappa\left(X\right)$, and positive definite if and only if $\kappa\left(X\right)$ is a linearly independent set.\medskip\qed\end{remark}

\textbf{Parallelepipeds. }For a finite subset $B=\left\{v_{1},...,v_{n}\right\}$ of a real Hilbert space $H$, we adopt the notations $-B=\left\{-v_{1},...,-v_{n}\right\}$ and $\sum B=\sum\limits_{i=1}^{n}v_{i}$. Also, for $v\in H$ define $v_{B}$ to be the projection of $v$ onto $\spa B$. Then $\|v-v_{B}\|$ is the distance from $v$ to the closed subspace $\spa B$ of $H$. Note that if $A\subset B$, then $\left(v_{B}\right)_{A}=v_{A}$, for every $v\in H$. Assume that $B=\left\{v_{1},...,v_{n}\right\}\subset H$ is linearly independent. The \emph{parallelepiped} defined by $B$ is a set $$P\left(v_{1},...,v_{n}\right)=P\left(B\right)=\left\{\sum_{i=1}^{n}\alpha_{i} v_{i}\left|\alpha_{i}\in\left[0,1\right]\right.\right\}$$ (we will also define $P\left(B\right)=\left\{0_{H}\right\}$ if $B$ is empty or linearly dependent). Then $P\left(B\right)$ is convex, and its extreme points (vertices) are $\left\{\sum A \left|A\subset B\right.\right\}$. Hence, there are $2^{n}$ vertices in total. Note that $\sum A$ is the vertex, located on the opposite ``end'' to the origin in the ``subparallelepiped'' $P\left(A\right)$, which is a face of $P\left(B\right)$. Hence, we have a correspondence between vertices and faces of $P\left(B\right)$. For example, $0_{H}$ corresponds to the face $P\left(\varnothing\right)$ and $P\left(B\right)$ itself corresponds to the vertex $\sum B$.

The set of edges of $P\left(B\right)$ that radiate from $\sum A$ are $-A\cup B\backslash A$. Hence, \linebreak $P\left(-A\cup B\backslash A\right)=P\left(B\right)-\sum A$. In essence, the shift $-\sum A$ accounts for change of the point of $H$ which we call the origin.

While the collection $B$ is reconstructible from $P\left(B\right)$ as a subset of $H$, it is not reconstructible from $P\left(B\right)$ as a geometric figure. Indeed, $P\left(-A\cup B\backslash A\right)$ is a parallel translation of $P\left(B\right)$, for any $A\subset B$, and if $T:H\to H$ is an isometric operator, then $P\left(TB\right)=TP\left(B\right)$ is isometric to $P\left(B\right)$. Below we will show that isometric parallelepipeds can be transformed into each other by a combination of a unitary operator on $H$ and a parallel translation.

Define $V\left(v_{1},...,v_{n}\right)=V\left(B\right)=V_{n}\left(P\left(B\right)\right)$, where $V_{n}$ is the $n$-dimensional Hausdorff measure (we will also adopt a convention that $V\left(\varnothing\right)=0$). Then $$V\left(B\right)^{2}=\det\left[\left<v_{i},v_{j}\right>\right]_{i,j=1}^{n}.$$ In particular, $V\left(v\right)=\|v\|$ and $V\left(v,w\right)=\sqrt{\|v\|^{2}\|w\|^{2}-\left|\left<v,w\right>\right|^{2}}$. Also note that for any $v\in H$ we have

$$\|v-v_{B}\|=\frac{V\left(B\cup \left\{v\right\}\right)}{V\left(B\right)}.$$
\

\textbf{Theorem \ref{similar} for real-valued positive semi-definite kernels. }Let $K$ and $L$ be non-degenerate positive semi-definite kernels over $X$ and let $\kappa:X\to H\backslash\left\{0_{H}\right\}$ and $\lambda:X\to E\backslash\left\{0_{E}\right\}$ be the maps into the Hilbert spaces provided by the Moore-Aronszajn theorem. We can describe the relation of being Hermitean rescalings in terms of $\kappa$ and $\lambda$. Namely, $L=f\otimes \overline{f} K$, for $f:X\to\Cp$, if and only if there is a unitary operator $T:H\to E$ such that $\lambda=f\cdot T\kappa$.

Now assume that $K$ and $L$ are real-valued (and so symmetric), and so $H$ and $E$ are real Hilbert spaces. Then, for any $x_{1},...,x_{n}\in X$ we have $\det_{K}\left(x_{1},...,x_{n}\right)=V\left(\kappa\left(x_{1}\right),...,\kappa\left(x_{n}\right)\right)^{2}$. Hence, if $\lambda=f\cdot \kappa$, for some $f:X\to\R$, it follows from Proposition \ref{pdc} that $\|\lambda\|=\left|f\right|\|\kappa\|$, and more generally

\begin{equation}\label{**}
V\left(\lambda\left(x_{1}\right),...,\lambda\left(x_{n}\right)\right)=\left|f\left(x_{1}\right)\right|...\left|f\left(x_{n}\right)\right|V\left(\kappa\left(x_{1}\right),...,\kappa\left(x_{n}\right)\right).\tag*{($\ast\ast$)}
\end{equation}

Also, $K$ and $L$ are $\pm1$ symmetric rescalings if and only if $\lambda\left(x\right)=\pm T\kappa\left(x\right)$, for every $x\in X$. We can now restate Theorem \ref{similar} in a geometric form.

\begin{theorem}\label{pdsimilar}
Let $H$ be a real Hilbert space and let $\kappa,\lambda:X\to H$ be such that $\overline{\spa~\kappa\left(X\right)}=H$. Then the following are equivalent:
\item[(i)] There is an isometry $T$ on $H$ such that $\lambda\left(x\right)=\pm T\kappa\left(x\right)$, for every $x\in X$;
\item[(ii)] For any (distinct) $x_{1},...,x_{n}\in X$ the parallelepipeds $P\left(\kappa\left(x_{1}\right),...,\kappa\left(x_{n}\right)\right)$ and \linebreak $P\left(\lambda\left(x_{1}\right),...,\lambda\left(x_{n}\right)\right)$ are isometric;
\item[(iii)] For any (distinct) $x_{1},...,x_{n}\in X$ the parallelepipeds $P\left(\kappa\left(x_{1}\right),...,\kappa\left(x_{n}\right)\right)$ and \linebreak $P\left(\lambda\left(x_{1}\right),...,\lambda\left(x_{n}\right)\right)$ have the same volume.
\end{theorem}
\begin{proof}
(i)$\Rightarrow$(ii) follows from the fact that for any linearly independent $v_{1},...,v_{n}\in H$ all parallelepipeds $P\left(\pm v_{1},...,\pm v_{n}\right)$ are parallel translations of each other. (ii)$\Rightarrow$(iii) is obvious and (iii)$\Rightarrow$(i) follows from Theorem \ref{similar} and the discussion above (applied to $H$ and $E=\overline{\spa~\lambda\left(X\right)}$).
\end{proof}

\begin{remark} Note that the counterexamples $L_{n}^{\pm}$ and $L_{A}$ considered in Example \ref{ex} are positive definite because Hermitean strictly diagonally dominant matrices with positive diagonal entries are positive definite (see \cite[Corollary 7.2.3]{hj}). In particular, $L_{3}^{\pm}$ correspond to vectors $u^{\pm},v^{\pm},w^{\pm}\in\R^{3}$ of length $2$ and such that the angles between $u^{\pm}$ and $w^{\pm}$ is $\arccos\frac{\pm 1}{4}$, while the angles between each of these vectors and $v^{\pm}$ is $\arccos\frac{1}{4}$. Thus, we need to check the equality of volumes of faces of all dimensions.\qed\end{remark}

We can also state a geometric version of Proposition \ref{pdcr}.

\begin{corollary}\label{pdcrg}
Let $H$ be a real Hilbert space and let $\kappa,\lambda:X\to H\backslash\left\{0_{H}\right\}$ be such that $\overline{\spa~\kappa\left(X\right)}=H$. Then there are an isometry $T$ on $H$ and $g:X\to\Rp$ such that $\lambda=g\cdot T\circ\kappa$ if and only if for any (distinct) $x_{1},...,x_{n}\in X$ we have $$\|\kappa\left(x_{1}\right)\|...\|\kappa\left(x_{n}\right)\|V\left(\lambda\left(x_{1}\right),...,\lambda\left(x_{n}\right)\right)=\|\lambda\left(x_{1}\right)\|...\|\lambda\left(x_{n}\right)\|V\left(\kappa\left(x_{1}\right),...,\kappa\left(x_{n}\right)\right).$$
\end{corollary}

\medskip

\textbf{Equality of parallelepipeds. }In the particular case when $X$ is finite, Theorem \ref{pdsimilar} turns into a criterion for equality of parallelepipeds.

\begin{theorem}\label{paral}
Two parallelepipeds with equal volumes of the corresponding faces of all dimensions are isometric.
\end{theorem}

Since the result above is a purely geometric fact it is desirable to have a geometric proof for it. Note that the remark above evidences that the theorem is not a corollary of Minkowski's theorem. We will present a proof which roughly follows the pattern of the proof of Theorem \ref{similar}, but operates with geometric objects instead of matrices and abstract bi-functions. Let $K$ be a positive semi-definite kernel again and let $H$ and $\kappa:X\to H$ be given by Moore-Aronszajn theorem.

Consider the decomposition $X=\bigsqcup\limits_{j\in I} X_{j}$ into the components of the graph $X_{K}$ from the geometric viewpoint. It is clear that if $i\ne j$, then $\kappa\left(X_{i}\right)\perp\kappa\left(X_{j}\right)$, and so $H=\bigoplus \limits_{j\in I} \overline{\spa \kappa\left(X_{j}\right)}$. Therefore, for any collection $\left\{\alpha_{j}\right\}_{j\in I}\subset\T$ the operator $T=\bigoplus \limits_{j\in I}\alpha_{j}Id_{\overline{\spa \kappa\left(X_{j}\right)}}$ is a unitary operator on $H$ such that $\kappa\left(X_{j}\right)$ belongs to the eigenspace of $T$ that corresponds to $\alpha_{j}$. Conversely, if for some $j\in I$ every element of $\kappa\left(X_{j}\right)$ is an eigenvector of an isometry $T:H\to H$, then the corresponding eigenvalues coincide. This observation is an analogue of Proposition \ref{components}.

Let us consider a special class of paths in $X_{K}$. We will call a finite sequence $v_{1},...,v_{n}\in H$ a \emph{chain} (from $v_{1}$ to $v_{n}$) if $v_{i}\perp v_{j}$ whenever $|i-j|>1$ and $v_{i}\not\perp v_{i+1}$, for every $i\in \overline{1,n-1}$. It is clear that if $x$ and $y$ are connected by a path in $X_{K}$, then $\kappa\left(x\right)$ is connected with $\kappa\left(y\right)$ via a chain of elements of $\kappa\left(X\right)$, which corresponds to a minimal path from $x$ to $y$. This configuration has the following property.

\begin{lemma}
If $v_{1},...,v_{n}\in H$ is a chain, then $v_{1},...,v_{n-1}$ are linearly independent.
\end{lemma}
\begin{proof}
We will show that $v_{k}\not\in \spa\{v_{1},...,v_{k-1}\}$ for every $1<k<n$. Assume that $v_{k}=\alpha_{1} v_{1}+...+\alpha_{k-1} v_{k-1}$, where $1<k<n$ and $\alpha_{1},...,\alpha_{k-1}\in \C$. Then each of $v_{1},v_{2},...,v_{k-1}$ are orthogonal to $v_{k+1}$, and so $v_{k}\perp v_{k+1}$, which contradicts the definition of chain.
\end{proof}

\begin{remark} It also follows that $v_{1}\not\in \spa\{v_{2},...,v_{n-1}\}$. Indeed, if $v_{1}=\alpha_{2} v_{2}+...+\alpha_{k} v_{k}$, and $\alpha_{k}\ne 0$, then $v_{k}\in \spa\{v_{1},...,v_{k-1}\}$, and so $k=n$.\qed\end{remark}

Chains help to estimate a dimension corresponding to the components of $X_{K}$, which is also utilized in the proof of the following lemma.

\begin{lemma}\label{chain}
Let $B\subset H$ be finite and linearly independent. Let $u,w\in \spa B$ be such that $u_{A}\perp w_{A}$, for any $A\subset B$. Then there are subsets $B_{u}$ and $B_{w}$ of $B$ such that $u\in \spa B_{u}, w\in\spa B_{w}$ and $B_{u}$, $B_{w}$ and $B\backslash \left(B_{u}\cup B_{w}\right)$ are mutually orthogonal.
\end{lemma}
\begin{proof}
Without loss of generality we may assume that $\spa B=H$.

Define $B_{u}$ to be the set of all $v\in B$ such that there is a chain $u_{0}=u,u_{1},...,u_{n}=v$, where $u_{1},...,u_{n}\in B$, and define $B_{w}$ analogously. It is clear that $u\perp B\backslash B_{u}$, $w\perp B\backslash B_{w}$ and $B_{u}\cup B_{w}\perp B\backslash \left(B_{u}\cup B_{w}\right)$. In order to prove the lemma it is enough to show that $B_{u}\perp B_{w}$, since then we would have $u\in \left(B\backslash B_{u}\right)^{\perp}=\spa B_{u}$ and analogously $w\in \spa B_{w}$.

We need to show that if $u_{1},...,u_{n},w_{1},...,w_{m}\in B$ are such that $u_{0}=u,u_{1},...,u_{n}$ and $w_{0}=w,w_{1},...,w_{m}$ are chains, then $u_{n}\perp w_{m}$. We will use the induction by $m+n$. When $m+n=0$ this follows from $u_{0}=u=u_{B}\perp w_{B}=w=w_{0}$.

Assume the claim holds for $m+n$ and let $A=\left\{u_{1},...,u_{n},u_{n+1},w_{1},...,w_{m}\right\}\subset B$ be such that $u,u_{1},...,u_{n},u_{n+1}$ and $w,w_{1},...,w_{m}$ are chains. Let $u_{0}=u_{A}$ and $w_{0}=w_{A}$. Then it is easy to see that $u_{0}\perp u_{i}$, for $i>2$, and $u_{0}\not\perp u_{1}$, and so $u_{0},u_{1},...,u_{n},u_{n+1}$ is a chain. Hence, by the preceding lemma $u_{0},u_{1},...,u_{n}$ are linearly independent. Analogously, $w_{0},w_{1},...,w_{m-1}$ are also linearly independent (note that this set is empty if $m=0$).

From the hypothesis of induction $u_{i}\perp w_{j}$, when $i\le n$, and so $\left\{u_{0},u_{1},...,u_{n}\right\}\perp\left\{w_{0},w_{1},...,w_{m}\right\}$. All these $m+n+2$ vectors belong to $\spa A$, whose dimension is $m+n+1$. Since $u_{0},u_{1},...,u_{n},w_{0},w_{1},...,w_{m-1}$ are linearly independent, it follows that $w_{m}\in \spa\{w_{0},w_{1},...,w_{m-1}\}$. As all of the vectors in the latter span are orthogonal to $u_{n+1}$, we conclude that $u_{n+1}\perp w_{m}$.
\end{proof}

\textbf{Geometric proof of Theorem \ref{paral}. }Let $\left\{v_{1},...,v_{n}\right\}$ and $\left\{w_{1},...,w_{n}\right\}$ be two linearly independent sets in a real Hilbert space. We will denote the fact that $V\left(v_{i_{1}},...,v_{i_{k}}\right)=V\left(w_{i_{1}},...,w_{i_{k}}\right)$, for any $\left\{i_{1},...,i_{k}\right\}\subset\left\{1,...,n\right\}$ by $\left\{v_{1},...,v_{n}\right\}\approx\left\{w_{1},...,w_{n}\right\}$. Note that $\left\{v_{1},...,v_{n}\right\}\approx\left\{\pm v_{1},...,\pm v_{n}\right\}$, for any $\left\{v_{1},...,v_{n}\right\}$, and any distribution of signs.

After all the preparatory work we have done, let us prove that if $\left\{v_{1},...,v_{n}\right\}\approx\left\{w_{1},...,w_{n}\right\}$, then there is a unitary operator $T:H\to H$ such that $Tw_{i}=\pm v_{i}$, for every $i\in\overline{1,n}$. The proof is done by induction over $n$. For $n=1$ the result holds trivially.

Assume that the claim is true for $n$ and let $\left\{v_{0},v_{1},...,v_{n}\right\}\approx\left\{w_{0},w_{1},...,w_{n}\right\}$. From the hypothesis of induction applied to $\left\{v_{1},...,v_{n}\right\}$ and $\left\{w_{1},...,w_{n}\right\}$ there is a unitary operator $S:H\to H$ such that $Sw_{i}=\pm v_{i}$ for $i\in\overline{1,n}$. Let $v'_{0}=Sw_{0}$. Then $$\left\{v_{0},v_{1},...,v_{n}\right\}\approx\left\{w_{0},w_{1},...,w_{n}\right\}\approx\left\{Sw_{0},Sw_{1},...,Sw_{n}\right\}\approx \left\{v'_{0},v_{1},...,v_{n}\right\},$$ and so for any $A\subset B=\left\{v_{1},...,v_{n}\right\}$ we have $$\|v_{0}-\left(v_{0}\right)_{A}\|V\left(A\right)=V\left(A\cup\left\{v_{0}\right\}\right)=V\left(A\cup\left\{v'_{0}\right\}\right)=\|v'_{0}-\left(v'_{0}\right)_{A}\|V\left(A\right).$$
Since we also have $\|v_{0}\|=V\left(\left\{v_{0}\right\}\right)=V\left(\left\{v'_{0}\right\}\right)=\|v'_{0}\|$, by Pythagoras theorem we conclude that $\|\left(v_{0}\right)_{A}\|=\|\left(v'_{0}\right)_{A}\|$.

Define $2u=\left(v_{0}\right)_{B}+\left(v'_{0}\right)_{B}$ and $2w=\left(v_{0}\right)_{B}-\left(v'_{0}\right)_{B}$. Then $\left(v_{0}\right)_{A}=u_{A}+w_{A}$ and $\left(v'_{0}\right)_{A}=u_{A}-w_{A}$, for any $A\subset B$. Therefore, $\|u_{A}+w_{A}\|=\|u_{A}-w_{A}\|$, and so $u_{A}\perp w_{A}$. By Lemma \ref{chain} there are subsets $B_{u}$ and $B_{w}$ of $B$ such that $u\in \spa B_{u}, w\in\spa B_{w}$ and $B_{u}$, $B_{w}$ and $B\backslash \left(B_{u}\cup B_{w}\right)$ are mutually orthogonal.

Since $v'_{0}-\left(v'_{0}\right)_{B}$ and $v_{0}-\left(v_{0}\right)_{B}$ are both orthogonal to $\spa B$, there is a unitary operator $Q$ on $H$ such that $Qv=v$, when $v\in B_{u}$, $Qv=-v$, when $v\in B\backslash B_{u}$ and $Q\left(v'_{0}-\left(v'_{0}\right)_{B}\right)=v_{0}-\left(v_{0}\right)_{B}$. Thus, $T=QS$ is a unitary operator on $H$ such that $Tw_{i}=QSw_{i}=Q\left(\pm v_{i}\right)=\pm v_{i}$, for $i\in\overline{1,n}$ and $$Tw_{0}=QSw_{0}=Qv'_{0}=T\left(v'_{0}-\left(v'_{0}\right)_{B}+u-w\right)=v_{0}-\left(v_{0}\right)_{B}+u+w=v_{0}.$$

\begin{remark} One may wonder if during the proof of the step of induction we can slightly deform $S$ and find a unitary operator $T$ such that $Tw_{i}=Sw_{i}$, for $i\in\overline{1,n}$ and $Tw_{0}=\pm v_{0}$. This is not always possible: consider $H=\R^{3}$ and let $v_{1}=w_{1}=\left[1,0,0\right]$, $v_{2}=w_{2}=\left[0,1,0\right]$, $v_{3}=\left[1,1,1\right]$ and $w_{3}=\left[1,-1,1\right]$. Then there is no isometry $T$ of $H$, which is identity on the first two coordinates and $Tw_{3}=\pm v_{3}$.\medskip\qed
\end{remark}

\textbf{A characterization of isometries on Hilbert spaces. }We conclude the paper with a continuous version of Theorem \ref{pdsimilar}. First, some remarks about continuous positive semi-definite kernels. It is easy to see from an observation in the beginning of Section \ref{restop}, and from Proposition \ref{conr} that if $L$ and $M$ are positive semi-definite with $L^{2}=M^{2}$, then $L=M$ provided that $L$ and $M$ are rescalings, or do not vanish. The following example shows that the claim is wrong without the extra assumptions.

\begin{example}
If in Example \ref{exa} $K_{0}$ was chosen to be positive definite, then the resulting $K^{\pm}$ are positive definite. Indeed, let $H$ be a Hilbert space and let $\kappa_{n}:\left(-1,1\right)\to H$ be produced from $K_{n}$ by Moore-Aronszajn theorem. Define $\kappa:\R\to \bigoplus\limits_{n\in\Z} H$ by $\kappa=\bigoplus\limits_{n\in\Z}\kappa_{n}$. Then, $K\left(x,y\right)=\left<\kappa\left(x\right),\kappa\left(y\right)\right>$, for $x,y\in \R$, and since $\kappa_{0}\left(\left(-1,1\right)\right)$ is linearly independent, one can easily deduce linear independence of $\kappa\left(\R\right)$. Thus, $K$ is positive definite. As an example of a positive definite kernel that satisfies the conditions of Example \ref{exa} one can take $$K_{0}\left(x,y\right)=\re\frac{\left(1-x^{2}\right)\left(1-y^{2}\right)}{4-e^{\ii\pi\left(x-y\right)}}= \frac{1-x^{2}}{2}\frac{1-y^{2}}{2}\re M\left(\frac{e^{\ii\pi x}}{2},\frac{e^{\ii\pi y}}{2}\right),$$ where $M$ is the classical Szego kernel.\qed
\end{example}

Let us finally characterize restrictions of isometries on weakly connected subsets of a Hilbert space.

\begin{theorem}\label{iso}
Let $H$ be a real Hilbert space and let $B\subset H\backslash\left\{0_{H}\right\}$ be connected in the weak topology of $H$ and such that $\overline{\spa~B}=H$. Let $\Phi:B\to H$ be continuous with respect to the weak topology and such that for any distinct $v_{1},...,v_{n}\in B$ the parallelepipeds $P\left(v_{1},...,v_{n}\right)$ and $P\left(\Phi\left(v_{1}\right),...,\Phi\left(v_{n}\right)\right)$ have the same volume. Then there is an isometry $T$ on $H$ such that $T\left|_{B}\right.=\Phi$.
\end{theorem}
\begin{proof}
Let $K$ and $L$ be bi-functions on $B$ defined by $K\left(v,w\right)=\left<v,w\right>$ and $L\left(v,w\right)=\left<\Phi\left(v\right),\Phi\left(w\right)\right>$. Clearly, both $K$ and $L$ are  non-degenerate and symmetric. From the definition of the weak topology $K$ is separately continuous with respect to the weak topology on $B$, but the same is also true for $L$, as $\Phi$ is continuous with respect to the weak topology.

Since parallelepipeds $P\left(v_{1},...,v_{n}\right)$ and $P\left(\Phi\left(v_{1}\right),...,\Phi\left(v_{n}\right)\right)$ have the same volume, it follows that $\det_{K}=\det_{L}$. Since $B$ endowed with the weak topology is a connected topological space, and $K$ and $L$ are non-degenerate and separately continuous, from the part (ii) of Proposition \ref{conr} we conclude that $K=L$. Hence, according to Moore-Aronszajn theorem there is a unitary operator $T: H\to \overline{\spa \Phi\left(B\right)}$, such that $Tv=\Phi\left(v\right)$, for every $v\in B$.
\end{proof}
\begin{remark}
According to Corollary \ref{comp}, in the event that $B$ is weakly compact there is $l\in\N$ such that it is enough to check the equality of the volumes of the parallelepipeds of dimension up to $l$.\medskip\qed
\end{remark}

The following example shows that we cannot characterize restrictions of linear contractions via the property of non-increasing of volumes of parallelepipeds.

\begin{example}
Let $X=\R$ and let $K\left(x,y\right)=\exp\left(-\left|x-y\right|\right)$, which is the reproducing kernel of the Sobolev space $H$ endowed with the norm $\|f\|^{2}=\frac{1}{2}\int\limits_{R} u^{2}\left(t\right)+\left[u'\right]^{2}\left(t\right)dt$ (see \cite[6.1.6.1, Corollary 16]{bt}), in the sense that the map $\kappa:\R\to H$ such that $\left<f,\kappa\left(x\right)\right>=f\left(x\right)$, for all $x\in \R$ and $f\in H$, also satisfies $\overline{\spa~\kappa\left(X\right)}=H$ and $K\left(x,y\right)=\left<\kappa\left(x\right),\kappa\left(y\right)\right>$, for every $x,y\in X$. One can show that $\kappa$ is a continuous injection from $\R$ into $H$, and so $B=\kappa\left(X\right)$ is connected.

Let $g:\R\to\left(0,1\right]$ be continuous, have a bounded derivative, and satisfy $g\left(0\right)=1$ and $g\left(\frac{1}{2}\right)=a<\frac{2}{\sqrt{e}}-1$. Then $T:H\to H$ defined by $Tf=fg$ is a bounded linear operator. For $x\in \R$ we have $\left<Tf,\kappa\left(x\right)\right>=f\left(x\right)g\left(x\right)=\left<f,\Phi\left(\kappa\left(x\right)\right)\right>$, where $\Phi:\kappa\left(\R\right)\to H$ is defined by $\Phi\left(\kappa\left(x\right)\right)=g\left(x\right)\kappa\left(x\right)$. Hence, $\Phi=\left.T^{*}\right|_{\kappa\left(X\right)}$ is a continuous map. Moreover, from \ref{**} it follows that $\Phi$ does not increase the volumes of parallelepipeds. However, $$\|\kappa\left(\frac{1}{2}\right)-\kappa\left(0\right)\|^{2}=\|\kappa\left(\frac{1}{2}\right)\|^{2}+\|\kappa\left(0\right)\|^{2}-2\left<\kappa\left(1\right),\kappa\left(0\right)\right>=2-\frac{2}{\sqrt{e}},$$ while $\|\Phi\left(\kappa\left(\frac{1}{2}\right)\right)-\Phi\left(\kappa\left(0\right)\right)\|^{2}=1+a^{2}-\frac{2a}{\sqrt{e}}$. Since $0<a<\frac{2}{\sqrt{e}}-1$ it follows that $2-\frac{2}{\sqrt{e}}<1+a^{2}-\frac{2a}{\sqrt{e}}$, and so $\Phi$ is not a contraction.\medskip

It is also possible for a map to be a contraction and a restriction of a bounded linear map on $H$, decrease the volumes of parallelepipeds, and yet not be a restriction of a linear contraction. Let $S:H\to H$ be defined by $\left[Sf\right]\left(x\right)=f\left(\frac{x}{2}\right)$. It is easy to see that $S$ is a bounded operator with $\|S\|\le 2$, but also $\|S\|>1$. In order to justify the second claim consider $f\in H$, which is $1$ on $\left[-1,1\right]$, $0$ on $\left[-\8,-2\right]\cup\left[2,+\8\right]$, and $3-\left|x\right|$ on $\left[-2,-1\right]\cup\left[1,2\right]$. Then $\|f\|^{2}=\frac{7}{3}$, while $\|Sf\|^{2}=\frac{11}{3}$. For every $x\in X$ we have $\left<Sf,\kappa\left(x\right)\right>=f\left(\frac{x}{2}\right)=\left<f,\Psi\left(\kappa\left(x\right)\right)\right>$, where $\Psi:\kappa\left(\R\right)\to \kappa\left(\R\right)$ is defined by $\Psi\left(\kappa\left(x\right)\right)=\kappa\left(\frac{x}{2}\right)$. Hence, $\Psi=\left.S^{*}\right|_{\kappa\left(X\right)}$, while it follows from $\overline{\spa~\kappa\left(X\right)}=H$, that if $\Psi=\left.P\right|_{\kappa\left(X\right)}$, for some bounded linear operator $P$ on $H$, then $P=S^{*}$.

However, $\Psi$ is a contraction and does not increase the volumes of parallelepipeds. Indeed,  $\|\Psi\left(\kappa\left(x\right)\right)-\Psi\left(\kappa\left(y\right)\right)\|^{2}=2-2e^{-\frac{\left|x-y\right|}{2}}< 2-2e^{-\left|x-y\right|}=\|\kappa\left(x\right)-\kappa\left(y\right)\|^{2}$, for any distinct $x,y\in \R$. Also, one can show that if $x_{1}<...<x_{n}$, then $$\det\nolimits_{K}\left(x_{1},...,x_{n}\right)=\left(1-e^{-2\left|x_{2}-x_{1}\right|}\right)...\left(1-e^{-2\left|x_{n}-x_{n-1}\right|}\right),$$
and so $\det_{K}\left(x_{1},...,x_{n}\right)>\det_{K}\left(\frac{x_{1}}{2},...,\frac{x_{n}}{2}\right)$.\medskip
\qed\end{example}

Let us conclude the article with two questions.

\begin{question}Let $H$ be a real Hilbert space with the unit sphere $S$, and let $B\subset S$ be connected in the weak topology of $H$ and such that $\overline{\spa~B}=H$. Let $\Phi:B\to S$ be continuous with respect to the weak topology and such that for any distinct $v_{1},...,v_{n}\in B$ the volume of parallelepiped $P\left(\Phi\left(v_{1}\right),...,\Phi\left(v_{n}\right)\right)$ does not exceed the volume of $P\left(v_{1},...,v_{n}\right)$. Is it true that $\Phi$ is a contraction, i.e. does not increase distances? If yes, is there is a contraction $T$ on $S$ such that $T\left|_{B}\right.=\Phi$?
\end{question}

\begin{question}For which Banach spaces $H$ Theorem \ref{iso} holds?
\end{question}

In order to mimic the proof of the theorem or to see where it fails we need to be able to calculate the volumes of parallelepipeds spanned by a finite set of vectors. Let $v_{1},...,v_{n}\in H$ be linearly independent, and let $\T:\R^{n}\to E=\spa\left\{v_{1},...,v_{n}\right\}$ be the linear transformation such that $Te_{k}=v_{k}$, for every $k\in\overline{1,n}$, where $e_{1},...,e_{n}$ are the standard basis vectors in $\R^{n}$. Then the pull-back $\mu$ of the $n$-dimensional Hausdorff measure on $E$ is translation-invariant and finite on compacts, and so is a scalar multiple of the Lebesgue measure $\lambda_{n}$. Hence,
 $$V\left(v_{1},...,v_{n}\right)=\mu\left(\left[0,1\right]^{n}\right)=\lambda_{n}\left(\left[0,1\right]^{n}\right)\frac{\mu\left(B\right)}{\lambda_{n}\left(B\right)}=\frac{\mu\left(B\right)}{\lambda_{n}\left(B\right)},$$
where $B=T^{-1}\overline{B}_{E}$. Note that $\mu\left(B\right)$ is equal to the the $n$-dimensional Hausdorff measure of the unit ball of the normed space $E$ with respect to the metric generated by the norm, and so is equal to $\lambda_{n}\left(\overline{B}_{\R^{n}}\right)$ (see \cite[Theorem 3.7.5]{thompson}). Thus, the only quantity left to find is $\lambda_{n}\left(B\right)$.

\section{Acknowledgements}

The author wants to thank: his supervisor Nina Zorboska for general guidance; John Urschel, who brought the author's attention to Question \ref{q} and the paper \cite{rkt} in particular; Alexandre Eremenko who brought the author's attention to the Calabi rigidity via \href{http://mathoverflow.net/}{MathOverflow} and Liviu Nicolaescu whose alternative proof of Theorem \ref{paral} see on \href{https://mathoverflow.net/questions/300099/parallelepiped-is-defined-by-the-volumes-of-its-faces}{MathOverflow}.

\end{document}